\documentclass[a4paper,oneside,reqno,final]{amsart}
\usepackage[utf8]{inputenc}
\usepackage[a4paper,margin=3cm]{geometry} 
\usepackage{bm}
\usepackage{amsmath}
\usepackage{amsthm}
\usepackage{amscd}
\usepackage{amssymb}
\usepackage{MnSymbol}
\usepackage{latexsym}
\usepackage{eucal}
\usepackage{dsfont}
\usepackage{mathtools}
\usepackage{enumitem}

\usepackage{verbatim}
\usepackage{graphicx}
\usepackage{float}
\usepackage{array,booktabs}

\usepackage[colorlinks,pdftex]{hyperref}
\hypersetup{
linkcolor=black,
citecolor=black,
pdftitle={},
pdfauthor={Franziska K\"uhn},
pdfkeywords={},
}

\renewcommand\labelenumi{\textup{(\roman{enumi})}}
\renewcommand\theenumi\labelenumi
\renewcommand\labelenumii{(\alph{enumii})}
\renewcommand\theenumii\labelenumii

\renewcommand\theenumii\labelenumii

\swapnumbers
\theoremstyle{theorem} \newtheorem{theorem}{Theorem}[section]
\theoremstyle{theorem} \newtheorem{lemma}[theorem]{Lemma}
\theoremstyle{theorem} \newtheorem{proposition}[theorem]{Proposition}
\theoremstyle{theorem} \newtheorem{corollary}[theorem]{Corollary}
\theoremstyle{definition} \newtheorem{definition}[theorem]{Definition}
\theoremstyle{remark} \newtheorem{remark}[theorem]{Remark}
\theoremstyle{definition}  \newtheorem{example}[theorem]{Example}
\theoremstyle{definition}  \newtheorem*{ack}{Acknowledgement}

\DeclareMathOperator \re {Re}

\DeclareMathOperator \id {id}

\DeclareMathOperator \supp {supp}

\newcommand{\I}{\mathds{1}}
\newcommand\floor[1]{\left\lfloor #1 \right\rfloor}

\newcommand\mc[1] {\mathcal{#1}}
\newcommand\mbb[1] {\mathds{#1}}

\newcommand{\eps}{\varepsilon}

\newcommand{\sint}{\strokedint}

\newcommand{\primo}{1\textsuperscript{o}}
\newcommand{\primolist}{\par\smallskip\textbf\primo\;\;}
\newcommand{\secundo}{2\textsuperscript{o}}
\newcommand{\secundolist}{\par\smallskip\textbf\secundo\;\;}
\newcommand{\tertio}{3\textsuperscript{o}}
\newcommand{\tertiolist}{\par\smallskip\textbf\tertio\;\;}
\newcommand{\quarto}{4\textsuperscript{o}}
\newcommand{\quartolist}{\par\smallskip\textbf\quarto\;\;}

\makeatletter
\@namedef{subjclassname@2020}{%
  \textup{2020} Mathematics Subject Classification}
\makeatother

\newcommand{\mathscr}{\mathcal}
\newcommand{\real}{\mathds{R}}
\newcommand{\Bscr}{\mathscr{B}}

\begin{document}

\title[Convolution inequalities for Besov and Triebel--Lizorkin spaces]{Convolution inequalities for Besov and Triebel--Lizorkin spaces, and applications to convolution semigroups}
\author[F.~K\"{u}hn]{Franziska K\"{u}hn}
\address[F.~K\"{u}hn]{TU Dresden, Fakult\"at Mathematik, Institut f\"{u}r Mathematische Stochastik, 01062 Dresden, Germany. E-Mail: \textnormal{franziska.kuehn1@tu-dresden.de}}

\author[R.L.~Schilling]{Ren\'e L.\ Schilling}
\address[R.L.~Schilling]{TU Dresden, Fakult\"at Mathematik, Institut f\"{u}r Mathematische Stochastik, 01062 Dresden, Germany. E-Mail: \textnormal{rene.schilling@tu-dresden.de}}

\begin{abstract}
    We establish convolution inequalities for Besov spaces $B_{p,q}^s$ and Triebel--Lizorkin spaces $F_{p,q}^s$. As an application, we study the mapping properties of convolution semigroups, considered as operators on the function spaces $A_{p,q}^s$, $A \in \{B,F\}$. Our results apply to a wide class of convolution semigroups including the Gau{\ss}--Weierstra{\ss} semigroup, stable semigroups and heat kernels for higher-order powers of the Laplacian $(-\Delta)^m$, and we can derive various caloric smoothing estimates.
\end{abstract}

\hfill\emph{\underline{Accepted for publication in Studia Mathematica}}  \par \bigskip

\maketitle

The convolution $f*g$ of two functions $f,g: \mbb{R}^n \to \mbb{R}$ is defined by
\begin{equation*}
	(f*g)(x) := \int f(x-y) g(y) \, dy, \quad x \in \mbb{R}^n,
\end{equation*}
provided that the integral exists. It is well known that the convolution $f*g$ inherits the regularity  of both $f$ and $g$, in the sense that if, say, $f \in C_c^k(\mbb{R}^n)$ and $g \in C_c^m(\mbb{R}^n)$, then $f*g \in C_c^ {k+m}(\mbb{R}^n)$. In this article, we study the regularizing properties of the convolution in the scales of Besov spaces $B_{p,q}^s$ and Triebel--Lizorkin spaces $F_{p,q}^s$. Having in mind that the parameters $s$ and $p$ describe the regularity, resp., integrability properties of $f \in A_{p,q}^s$, $A \in \{B,F\}$, it is natural to expect that the implication
\begin{equation}
	f \in A_{p_1,q_1}^s, \; g \in A_{p_2,q_2}^u \implies f*g \in A_{p,q}^{s+u}, \quad\quad s,u \in \mbb{R}, \label{intro-eq1}
\end{equation}
holds under suitable assumptions on $p,p_1,p_2 \in [1,\infty]$ and $q,q_1,q_2 \in (0,\infty]$. We establish inequalities of the form
\begin{equation}
	\|f*g \mid A_{p,q}^s\| \leq \|f \mid A_{p_1,q}^s\| \cdot \|g \mid L_{p_2}\|
	\quad \text{and} \quad
	\|f*g \mid A_{p,q}^{s+u}\| \leq C \|f \mid A_{p_1,q_1}^s\| \cdot \|g \mid A_{p_2,q_2}^u\|,
	\label{intro-eq2}
\end{equation}
for $s\in\mbb{R}, u\geq 0$ and parameters  $p,p_i \in [1,\infty]$ and $q,q_i \in (0,\infty]$ satisfying  certain relations, see Section~\ref{conv} below; as a by-product we get \eqref{intro-eq1}. For Besov spaces ($A=B$), the convolution estimates \eqref{intro-eq2} have been shown by Burenkov \cite{burenkov90};  the proof relies heavily on Young's inequality. For Triebel--Lizorkin spaces, the proof is more complicated and, in particular, the case that (at least) one of the integrability parameters is infinite, i.e.\ $\max\{p,p_1,p_2\}=\infty$, requires some care.

In the second part of the paper, we use these convolution estimates to study the mapping properties of convolution semigroups, considered as operators on the spaces $A_{p,q}^s$. Section~\ref{semi} is about semigroups $P_t f = f*p_t$ with possibly signed kernels $(p_t)_{t \geq 0} \subseteq L_1(dx)$. It is immediate from \eqref{intro-eq2} that
\begin{align*}
	\|P_t f \mid A_{p,q}^s\| \leq \|p_t \mid L_1\| \cdot \|f \mid A_{p,q}^s\|
	\quad \text{and} \quad
	\|P_t f \mid A_{p,q}^{s+u}\| \leq C \|p_t \mid A_{1,\infty}^u\| \cdot \|f \mid A_{p,q}^{s}\|,
\end{align*}
and so the norm $\|p_t \mid A_{1,\infty}^u\|$ plays a central role. We establish several results which are useful to bound this norm. In particular, we show that the kernels $p_t$ are self-regularizing in time, cf.\ Corollary~\ref{semi-5}, and that $\|p_t \mid A_{1,\infty}^u\|$ can be bounded from above by a power of $\int |\nabla p_t(x)| \, dx$ provided that $p_t \in C^1$, see Theorem~\ref{semi-11}. In Section~\ref{levy} we turn to convolution semigroups with non-negative kernels $p_t$ satisfying $\int p_t(x) \, dx=1$. Such semigroups appear naturally in probability theory, where $(p_t)_{t>0}$ can be interpreted as the family of transition densities of a L\'evy process, i.e.\ a stochastic process with independent and stationary increments. We obtain general convolution estimates using gradient estimates from \cite{ssw12}. Moreover, we study the smoothing properties of subordinate semigroups and recover some results from the recent paper \cite{knop19}.
Our convolution estimates apply to a wide class of semigroups, including all semigroups which are generated by arbitrary powers of the Laplacian $(-\Delta)^\kappa$, $\kappa\in (0,\infty)$, i.e.\ the classical ($\kappa = 1$) and the generalized (higher-order, $\kappa >1$)  Gau{\ss}--Weierstra{\ss} semigroups (cf.\ Theorem~\ref{semi-11}), the Cauchy--Poisson semigroup ($\kappa=1/2$, cf.\ Example~\ref{levy-8}) and stable semigroups ($\kappa\in (0,1)$, cf.\ Corollary~\ref{levy-9}).

\section{Function spaces} \label{fun}

For $0< p \leq \infty$ we denote by $L_p(\mu)$ the space of functions which are $p$th order integrable with respect to the measure $\mu$ on $(\real^n,\Bscr(\real^n))$; note that $L_p(\mu)$ is a quasi-Banach space if $0<p<1$. The (quasi-)norm on the (quasi-)Banach spaces $L_p(\mu)$ is denoted by $\|\bullet \mid L_p(\mu)\|$. We simply write $L_p$ if $\mu$ is Lebesgue measure $\lambda$, and $\ell_p:=\ell_p(\mbb{N}_0)$ is the space of $p$th order summable sequences indexed by $\mbb{N}_0$. For a sequence of functions $f_k: \mbb{R}^n \to \mbb{R}$, $k \geq 0$, we write
\begin{equation*}
		\|f_k \mid \ell_q \mid L_p(\mu)\|
	:=\|f_k(x) \mid \ell_q \mid L_p(\mu)\|
	:= \left( \int \left[ \sum_{k = 0}^{\infty} |f_k(x)|^q \right]^{p/q} \, \mu(dx) \right)^{1/p}
\end{equation*}
(with the usual modification if $p=\infty$ or $q=\infty$); analogously, $\|f_k \mid L_p(\mu) \mid \ell_q\|$ means that we first take the $L_p(\mu)$-norm (in the variable $x$) and then the $\ell_q$-norm (in $k$).

For Schwartz functions $\phi \in \mc{S}(\mbb{R}^n)$, say, we denote by
\begin{equation*}
		\mc{F}\phi(\xi) := \frac{1}{(2\pi)^{n/2}} \int_{\mbb{R}^n} e^{-ix \xi} \phi(x) \, dx, \quad \xi \in \mbb{R}^n,
\end{equation*}
the Fourier transform; the inverse Fourier transform is denoted by $\mc{F}^{-1}\phi$. Both the Fourier transform and its inverse are extended in the usual way to the space of tempered distributions $\mc{S}'(\mbb{R}^n)$. We denote by
\begin{equation*}
	\phi(D) f(x) := \mc{F}^{-1}(\phi \cdot \mc{F}f)(x)
\end{equation*}
the \emph{pseudo-differential operator} (Fourier multiplier operator) with symbol $\phi$.

For some fixed $\phi_0 \in C_c^{\infty}(\mbb{R}^n)$ with $\I_{B(0,1)} \leq \phi_0 \leq \I_{B(0,3/2)}$, define $\phi_k(x) :=\phi_0(2^{-k}x)-\phi_0(2^{-(k-1)}x)$, $k \in \mbb{N}$, then $\sum_{k=0}^{\infty} \phi_k(x)=1$ for all $x \in \mbb{R}^n$, and so $(\phi_k)_{k \in \mbb{N}_0}$ is a \emph{resolution of unity}. Following Triebel \cite{triebel1,triebel4}, we define the Besov spaces $B_{p,q}^s$ and the Triebel--Lizorkin spaces $F_{p,q}^s$ with $s \in \mbb{R}$ by
\begin{align*}
    B_{p,q}^s
    &:= \left\{f \in \mc{S}'(\mbb{R}^n); \:\|f \mid B_{p,q}^s\| := \|2^{ks} \phi_k(D) f \mid L_p \mid \ell_q\| < \infty \right\}, \quad p,q \in (0,\infty], \\
	F_{p,q}^s
&:= \left\{ f \in \mc{S}'(\mbb{R}^n); \: \|f \mid F_{p,q}^s\| := \|2^{ks} \phi_k(D) f \mid \ell_q \mid L_p\| < \infty \right\}, \quad p \in (0,\infty), \; q \in (0,\infty].
\end{align*}
In order to define the Triebel--Lizorkin spaces $F_{p,q}^s$ for $p=\infty$, we use cubes $Q_{J,M} := 2^{-J} M + (0,2^{-J})^n$ and write $\sint_Q g(x) \, dx := {\lambda(Q)}^{-1} \int_Q g(x) \, dx$. Following \cite{triebel4} we set
\begin{align*}
	F_{\infty,q}^s
    &:= \bigg\{f \in \mc{S}'(\mbb{R}^n); \:\|f \mid F_{\infty,q}^s\|
    := \sup_{J \geq 0, M \in \mbb{Z}^n}  \left( \sint_{Q_{J,M}} \sum_{k=J}^{\infty} |2^{ks} \phi_k(D) f(x)|^q \, dx \right)^{1/q} < \infty \bigg\}, \,\, q \in (0,\infty), \\
	F_{\infty,\infty}^s
    &:= B_{\infty,\infty}^s.
\end{align*}
These definitions do not depend on the particular resolution of unity $(\phi_k)_{k \in \mbb{N}_0}$: Different resolutions of unity lead to equivalent (quasi-)norms, cf.\ \cite[Proposition 2.3.2]{triebel1} and \cite[Remark 1.2]{triebel4}. A close inspection\footnote{To wit: \label{blind-label} For a sequence $(\phi_k)_{k \in \mbb{N}_0}$ with the above properties, set $m := \sum^{\infty}_{k=0} \phi_k \in C_b^{\infty}$. Then $m$ and $1/m$ are Fourier multipliers, see \cite[Section 1.5.2]{triebel1} (B-scale), \cite[Theorem, p.~75]{triebel1} ($F$-scale, $p<\infty$) and \cite{park} ($F$-scale, $p=\infty$), and it follows that both $(\phi_k)_{k \in \mbb{N}_0}$ and the resolution of unity $(\frac{1}{m} \phi_k)_{k \in \mbb{N}_0}$ define equivalent (quasi-)norms.} of the proof of the equivalence of the (quasi-)norms reveals that the equivalence holds for any sequence $(\phi_k)_{k \in \mbb{N}_0} \subseteq C_c^{\infty}(\mbb{R}^n)$ of non-negative functions with the following properties: $\supp \phi_0 \subseteq B(0,2)$, $\supp \phi_k \subseteq \{y \in \mbb{R}^n;\: 2^{k-1} \leq |y|<2^{k+1}\}$, \begin{equation*}
	\forall \alpha \in \mbb{N}_0^n\::\: \sup_{k \in \mbb{N}_0} \sup_{x \in \mbb{R}^n} 2^{k|\alpha|} |\partial^{\alpha} \phi_k(x)| < \infty,
\end{equation*}
and there are constants $0<c < C < \infty$ such that
\begin{equation*}
	0 < c \leq \sum_{k=0}^{\infty} \phi_k(x) \leq C, \quad x \in \mbb{R}^n.
\end{equation*}

\section{Convolution estimates for Besov and Triebel--Lizorkin spaces} \label{conv}

In this section, we establish convolution estimates of the form
\begin{equation}
	\|f*g \mid A_{p,q}^s\| \leq \|f \mid A_{p_1,q}^s\| \cdot \|g \mid L_{p_2}\|
	\quad \text{and} \quad
	\|f*g \mid A_{p,q}^{s+u}\| \leq C \|f \mid A_{p_1,q_1}^s\| \cdot \|g \mid A_{p_2,q_2}^u\| \label{conv-eq1}
\end{equation}
for Besov spaces ($A=B$) and Triebel--Lizorkin spaces ($A=F$); here, $s,u \in \mbb{R}$ are arbitrary real numbers while $p,p_1,p_2 \in [1,\infty]$ and $q,q_1,q_2 \in (0,\infty]$ have to satisfy certain relations, see Theorem~\ref{conv-1} and Theorem~\ref{conv-3} below for precise statements.

Let us recall some results from measure theory which play an important role in the proofs and which give at the same time some intuition about the inequalities \eqref{conv-eq1}. One of the key tools is Young's inequality for convolutions \cite[Problem 15.14]{mims}; it shows how the integrability of the convolution $f*g$ relates to the integrability of $f$ and $g$, namely,
\begin{equation}\label{young}
	\|f*g \mid L_p\| \leq \|f \mid L_{p_1}\| \cdot \|g \mid L_{p_2}\|
\end{equation}
for $p,p_1,p_2 \in [1,\infty]$ satisfying the relation\footnote{here, and in the sequel, we agree that $1/\infty = 0$.}
\begin{equation}\label{young2}
	1+\frac{1}{p} = \frac{1}{p_1} + \frac{1}{p_2};
\end{equation}
this shows, in particular, that $|f*g|<\infty$ Lebesgue almost everywhere for $f \in L_{p_1}$, $g \in L_{p_2}$ with $\tfrac{1}{p_1}+\frac{1}{p_2}\geq 1$. For the proof of \eqref{conv-eq1} in the $F$-scale, we will also need Minkowski's integral inequality, which states that
\begin{equation}\label{min}
	\left( \int_Z \left( \int_Y |v(y,z)| \, \mu(dy) \right)^q \nu(dz) \right)^{1/q}
	\leq \int_Y \left( \int_Z |v(y,z)|^q \, \nu(dz) \right)^{1/q} \mu(dy)
\end{equation}
for any measurable function $v:Y \times Z \to \mbb{R}$, any $q \in [1,\infty)$ and any two $\sigma$-finite measures $\mu$, $\nu$ on measurable spaces $(Y,\mc{Y})$ and $(Z,\mc{Z})$, respectively, cf.\ \cite[Theorem 14.16]{mims}. We will frequently use this inequality for the counting measure $\nu=\sum_{k=0}^{\infty} \delta_k$ on $Z=\mbb{N}_0$:
\begin{equation}
	\left(\sum_{k=0}^{\infty} \left[ \int_Y |v_k(y)| \, \mu(dy) \right]^{q} \right)^{1/q}
	\leq \int_Y \left( \sum_{k=0}^{\infty} |v_k(y)|^q \right)^{1/q} \, \mu(dy).
	\label{min-counting}
\end{equation}
We are now ready to state the first convolution estimate.

\begin{theorem} \label{conv-1}
	Let $A \in \{B,F\}$, $s \in \mbb{R}$, $q \in (0,\infty]$ \textup{(}resp.\ $q \in [1,\infty]$ if $A=F$\textup{)} and $p,p_1,p_2 \in [1,\infty]$ be such that
	\begin{equation}
		1+ \frac{1}{p} = \frac{1}{p_1}+ \frac{1}{p_2}. \label{conv-eq11}
	\end{equation}
	If $f \in A_{p_1,q}^s$ and $g \in L_{p_2}$, then $f*g \in A_{p,q}^s$ and
	\begin{equation}
		\|f*g \mid A_{p,q}^s\| \leq C \|f \mid A_{p_1,q}^s\| \cdot \|g \mid L_{p_2}\| \label{conv-eq12}
	\end{equation}
	where $C=1$ if $p_1<\infty$ and $C=2^n$ if $p_1=\infty$.
\end{theorem}

For Besov spaces (i.e.\ $A=B$), the convolution estimate is a consequence of Young's inequality; this has already been observed by Burenkov \cite{burenkov90}. For Triebel--Lizorkin spaces, the proof of Theorem~\ref{conv-1} is more delicate, especially if $p=\infty$ or $p_i =\infty$, and it relies on Minkowski's integral inequality. Let us add that Theorem~\ref{conv-1} -- and Theorem~\ref{conv-3} below -- also hold for so-called \emph{homogeneous Besov and Triebel--Lizorkin spaces}, see e.g.\ \cite{triebel1} for more information on these function spaces.

Let us point out a subtle point in the statement of Theorem~\ref{conv-1}: For $s>0$ and $f \in A_{p_1,q}^s$, $g \in L_{p_2}$, the existence of the convolution $f*g$ as an element of $L_p$ follows from $A_{p_1,q}^s \subseteq L_{p_1}$ and Young's inequality; if $s \leq 0$, the inclusion $A_{p_1,q}^s \subseteq L_{p_1}$ breaks down -- we have only $A_{p_1,q}^s \subseteq \mc{S}'(\mbb{R}^n)$ -- and it is a priori not clear whether the convolution $f*g$ is well defined; we will see in the course of the proof in which sense the convolution can be defined.

\begin{proof}[Proof of Theorem~\ref{conv-1}]
\primolist In this first part of the proof, we additionally assume that $f$ or $g$ is a Schwartz function. Since $L_p$ and $A_{p,q}^s$ are continuously embedded into $\mc{S}'(\mbb{R}^n)$, the convolution $f*g \in C^{\infty}(\mbb{R}^n)$ exists as element of $\mc{S}'(\mbb{R}^n)$;  moreover,
\begin{equation*}
    \mc{F}^{-1}(\phi_k \mc{F}(f*g))
    = (2\pi)^{n/2} \mc{F}^{-1}(\phi_k \mc{F}f \cdot \mc{F}g)
    = (\mc{F}^{-1}(\phi_k \mc{F}f))* g
\end{equation*}
for the dyadic resolution $(\phi_k)_{k \in \mbb{N}_0}$, i.e.
\begin{equation}\label{conv-eq15}
	\phi_k(D)(f*g)= (\phi_k(D)f)*g.
\end{equation}
For the remaining part of this step, we will tacitly assume that $q<\infty$; for $q=\infty$ the proof is similar and, in fact, somewhat easier.

\medskip
Case 1: $A=B$. From \eqref{conv-eq15} and Young's inequality \eqref{young}, we see that
\begin{align*}
	\|f*g \mid B_{p,q}^s\|
	= \|2^{ks} \phi_k(D)(f*g) \mid L_p \mid \ell_q\|
	&= \|2^{ks} (\phi_k(D)f)*g \mid L_p \mid \ell_q\| \\
	&\leq \|g \mid L_{p_2}\| \cdot \|2^{ks} \phi_k(D)f \mid L_{p_1} \mid \ell_q\| \\
	&= \|g \mid L_{p_2}\| \cdot \|f \mid B_{p_1,q}^s\|.
\end{align*}

\medskip
Case 2: $A=F$ and $p,p_1<\infty$. We start with the observation that
\begin{equation*}
	\|f_k * g \mid \ell_q\| \leq (\|f_k \mid \ell_q\|)* g
\end{equation*}
for every $g \geq 0$ and every sequence of non-negative functions $(f_k)_{k \in \mbb{N}}$ with $f_k(x) \in \ell_q$, $x \in \mbb{R}^n$. (Note that this inequality is trivial if $q=\infty$). This follows from Minkowski's integral inequality \eqref{min-counting} where we use Lebesgue measure $\mu=\lambda$ and set $v_k(y):= f_k(x-y) g(y)$ keeping $x \in \mbb{R}^n$ fixed. Together with \eqref{conv-eq15}, we thus get
\begin{align*}
	\|f*g \mid F_{p,q}^s\|
	= \|2^{ks} (\phi_k(D)f)*g \mid \ell_q \mid L_{p}\|
	&\leq \|2^{ks} |\phi_k(D)f|*|g| \mid \ell_q \mid L_p\| \\
	&\leq \big\|(\|2^{ks} \phi_k(D)f \mid \ell_q \|)*|g| \mid L_p\big\|.
\end{align*}
Now an application of Young's inequality gives the assertion.

\medskip
Case 3: $A=F$ and $p_1=\infty$. If $p_1=\infty$, then, by \eqref{conv-eq11}, $1+{1}/{p} = {1}/{p_2}$, which implies $p =\infty$ and $p_2=1$, i.e.\ we need to show that
\begin{equation*}
	\|f*g \mid F_{\infty,q}^s\| \leq \|f \mid F_{\infty,q}^s\| \cdot \|g \mid L_1\|.
\end{equation*}
As in the definition of the Triebel--Lizorkin space $F_{\infty,q}^s$, we set $Q_{J,M} := 2^{-J}M + (0,2^{-J})^n$ for $J \geq 0$, $M \in \mbb{Z}^n$ and write $\sint_Q f(x) \, dx := {\lambda(Q)}^{-1} \int_Q f(x) \, dx$. Using \eqref{conv-eq15} and monotone convergence, we get
\begin{align*}
	\|f*g \mid F_{\infty,q}^s\|
	&= \sup_{J \geq 0, M \in\mbb{Z}^n} \left( \sint_{Q_{J,M}} \sum_{k=J}^{\infty} |2^{ks} \phi_k(D)(f*g)(x)|^q \, dx \right)^{1/q} \\
	&= \sup_{J \geq 0, M \in \mbb{Z}^n} \left( \sum_{k=J}^{\infty} 2^{ksq} \sint_{Q_{J,M}} |(\phi_k(D)f)*g(x)|^q \, dx \right)^{1/q}.
\end{align*}
Jensen's inequality shows that
\begin{align*}
	|(\phi_k(D)f)*g|^q(x) \leq \|g \mid L_1\|^{q-1} \int_{\mbb{R}^n} |\phi_k(D)f|^q(x-y) \cdot |g(y)| \, dy,
\end{align*}
which implies by Tonelli's theorem
\begin{align*}
	\sint_{Q_{J,M}} |(\phi_k(D)f)*g|^q(x) \, dx \leq \|g \mid L_1\|^{q} \sup_{y \in \mbb{R}^n} \sint_{y+Q_{J,M}} |\phi_k(D)f|^q(x) \, dx.
\end{align*}
Combining this estimate with our identity for $\|f*g \mid F_{\infty,q}^s\|$ and using the fact that
\begin{align} \label{conv-eq19}
	\sup_{J \geq 0, M \in \mbb{Z}^n} \left( \sum_{k=J}^{\infty} 2^{ksq} \sup_{y \in \mbb{R}^n} \sint_{y+Q_{J,M}} |\phi_k(D)f|^q(x) \, dx \right)
	\leq 2^n \|f \mid F_{\infty,q}^s\|,
\end{align}
cf.\ \cite[proof of Theorem 2.3]{knop19}, the assertion follows.

\medskip
Case 4: $A=F$ and $p=\infty$. By \eqref{conv-eq11}, $p_1$ and $p_2$ are conjugate exponents ($p_1' = p_2$, for short), i.e.\ we have to show that
\begin{equation*}
	\|f*g \mid F_{\infty,q}^s\| \leq \|f \mid F_{p_1,q}^s\| \cdot \|g \mid L_{p_1'}\|.
\end{equation*}
For $p_1 = \infty$ we have already proved this in Case 3, and so we may assume $p_1<\infty$. Using Minkowski's integral inequality \eqref{min-counting}, we find that
\begin{align*}
	\left( \sum_{k=J}^{\infty} 2^{ksq} |(\phi_k(D)f)*g|^q(x) \right)^{1/q}
	\leq \int_{\mbb{R}^n} |g(y)| \underbrace{\left(\sum_{k=0}^{\infty} 2^{ksq} |\phi_k(D)f|^q(x-y) \right)^{1/q}}_{=:F(x-y)} \, dy.
\end{align*}
Since an application of H\"older's inequality yields
\begin{equation*}
	(|g|*F)(x)
	\leq \|F \mid L_{p_1}\| \cdot \|g \mid L_{p_1'}\|
	= \|f \mid F_{p_1,q}^s\| \cdot \|g \mid L_{p_1'}\|,
    \quad x \in \mbb{R}^n,
\end{equation*}
we conclude that
\begin{align*}
	\left( \sint_Q \sum_{k=J}^{\infty} 2^{ksq} |\phi_k(D)(f*g)(x)|^q \, dx \right)^{1/q}
	\leq \left( \sint_Q (|g|*F)^q(x) \, dx \right)^{1/q}
	\leq  \|f \mid F_{p_1,q}^s\| \cdot \|g \mid L_{p_1'}\|
\end{align*}
for any dyadic cube $Q \in \mc{B}(\mbb{R}^n)$. Taking the supremum over all cubes $Q$ and all $J \geq 0$ proves the desired inequality.

\secundolist We know from the first part of the proof that \eqref{conv-eq12} holds under the additional assumption that $f$ or $g$ is a Schwartz function, i.e.
\begin{equation}\label{conv-eq13}
	\|f*g \mid A_{p,q}^s\|
	\leq C \|f \mid A_{p_1,q}^s\| \cdot \|g \mid L_{p_2}\|,
	\quad f \in A_{p_1,q}^s,\; g \in L_{p_2}, \; \text{$f \in \mc{S}$ or $g \in \mc{S}$},
\end{equation}
where $C=1$ if $p_1<\infty$ and $C=2^n$ if $p_1=\infty$. We will now show that this implies \eqref{conv-eq12} for all $f,g$ satisfying the assumptions of the theorem.

Fix $f \in A_{p_1,q}^s$ and $g \in L_{p_2}$. By \eqref{conv-eq11}, $p_1<\infty$ or $p_2< \infty$, and so one of the following cases occurs.

\medskip
Case 1: $q<\infty$, $p_2\leq\infty$ and $p_1<\infty$. Then the space of Schwartz functions $\mc{S}(\mbb{R}^n)$ is dense in $A_{p_1,q}^s$, i.e.\ there is a sequence $(f_k)_{k \in \mbb{N}} \subseteq \mc{S}(\mbb{R}^n)$ with $f_k \to f$ in $A_{p_1,q}^s$. From \eqref{conv-eq13}, we see that
\begin{equation*}
	\|f_k *g - f_j * g \mid A_{p,q}^s\|
	= \|(f_k-f_j)*g \mid A_{p,q}^s\|
	\leq C \|f_k -f_j \mid A_{p_1,q}^s\| \cdot \|g \mid L_{p_2}\|,
\end{equation*}
and so $(f_k*g)_{k \in \mbb{N}}$ is a Cauchy sequence in $A_{p,q}^s$. By completeness, the limit $f*g := \lim_{k \to \infty} f_k*g$ exists in $A_{p,q}^s$. Since $A_{p,q}^s$ is continuously embedded into $\mc{S}'(\mbb{R}^n)$, it follows that $f_k*g \to f*g$ in $\mc{S}'(\mbb{R}^n)$ and it is easy to see that the limit does not depend on the approximating sequence $(f_k)_{k \in \mbb{N}}$. By \eqref{conv-eq13},
\begin{equation*}
	\|f_k*g \mid A_{p,q}^s\| \leq C \|f_k \mid A_{p_1,q}^s\| \cdot \|g \mid L_{p_2}\|,
\end{equation*}
and letting $k \to \infty$ proves the convolution estimate  for $f$ and $g$. \par

\medskip
Case 2:  $q \leq \infty$, $p_2<\infty$ and $p_1 \leq \infty$. Then we use that $\mc{S}(\mbb{R}^n)$ is dense in $L_{p_2}$, choose an approximating sequence $(g_k)_{k \in \mbb{N}} \subseteq \mc{S}(\mbb{R}^n)$, $g_k \to g$ in $L_{p_2}$, and argue as in Case 1.

\medskip
Case 3: $q=\infty$, $p_2= \infty$, $p_1<\infty$, and $A=F$. From \eqref{conv-eq11} we see that $p_1=1$ and $p=\infty$, i.e.\ we need to prove the convolution estimate for $f \in F_{1,\infty}^s$ and $g \in L_{\infty}$.

Fix $\eps>0$, then $f \in F_{1,\infty}^s \subseteq F_{1,\theta}^{s-\eps}$ for any $\theta>0$, and so there is a sequence $(f_i)_{k \in \mbb{N}} \subseteq \mc{S}(\mbb{R}^n)$ such that $f_i \to f$ in $F_{1,\theta}^{s-\eps} \hookrightarrow\mc{S}'(\mbb{R}^n)$. Thus, by Case 1, $f*g = \lim_{i \to \infty} f_i*g$ exists in $F_{\infty,\theta}^{s-\eps} \hookrightarrow \mc{S}'(\mbb{R}^n)$.\footnote{This idea to reduce things to a finite $p_2$ at the price of a slightly smaller smoothness index is from \cite[Proof of Theorem 2]{burenkov90}.} With exactly the same argument as in Step \primo, Case 4, we find for any $N \in \mbb{N}$ and any cube $Q \in \mc{B}(\mbb{R}^n)$ that
\begin{equation*}
	\left( \sint_{Q} \sum_{k=J}^N 2^{k(s-\eps)\theta} |\phi_k(D)(f_i*g)(x)|^{\theta} \, dx \right)^{1/\theta}
	\leq \| 2^{k(s-\eps)} \phi_k(D) f_i \mid \ell_{\theta}(0;N) \mid L_1\| \cdot \|g \mid L_{\infty}\|;
\end{equation*}
here $\|a_k \mid \ell_{\theta}(J;N)\| := \left( \sum_{k=J}^N |a_k|^{\theta} \right)^{1/\theta}$. As $f_i \to f$ in $F_{1,\theta}^{s-\eps}$ and $f_i*g \to f*g$ in $F_{\infty,\theta}^{s-\eps}$,  we can let $i \to \infty$ to get
\begin{equation}\label{conv-eq18}
	\left( \sint_{Q} \sum_{k=J}^N 2^{k(s-\eps)\theta} |\phi_k(D)(f*g)(x)|^{\theta} \, dx \right)^{1/\theta}
	\leq \| 2^{k(s-\eps)} \phi_k(D) f \mid \ell_{\theta}(0;N) \mid L_1\| \cdot \|g \mid L_{\infty}\|.
\end{equation}
Noting that \begin{align}\label{conv-eq20}
	\|2^{k(s-\eps)} \phi_k(D) f(x)\mid \ell_{\theta}(0;N)\|
	&\leq \|2^{ks} \phi_k(D) f(x)\mid \ell_{\theta}(0;N)\| \notag \\
	&\leq (N+1) \|2^{ks} \phi_k(D) f(x)\mid \ell_{\infty}\|  \in L^1(dx),
\end{align}
an application of Fatou's lemma on the left-hand side of \eqref{conv-eq18} and the dominated convergence theorem on the right-hand side of \eqref{conv-eq18} yield as $\epsilon\to 0$
\begin{align*}
	\left( \sint_{Q} \sum_{k=J}^N 2^{ks\theta} |\phi_k(D)(f*g)(x)|^{\theta} \, dx \right)^{1/\theta}
	&\leq \liminf_{\eps \to 0} \| 2^{k(s-\eps)} \phi_k(D) f \mid \ell_{\theta}(0;N) \mid L_1\| \cdot \|g \mid L_{\infty}\| \\
	&= \| 2^{ks} \phi_k(D) f \mid \ell_{\theta}(0;N) \mid L_1\| \cdot \|g \mid L_{\infty}\|<\infty.
\end{align*}
Recall that $\|v \mid L_{\infty}(\mu)\| = \lim_{\theta \to \infty} \|v \mid L_{\theta}(\mu)\|$ for any measure $\mu$ and any $v$ such that $v \in L_{\theta}(\mu)$ for large $\theta \gg 1$, see e.g.\ \cite[Exercise 13.21]{mims}; this allows us to let $\theta \to \infty$ on the left-hand side of the previous inequality. For the right-hand side we use once more the dominated convergence theorem, which is applicable because of \eqref{conv-eq20}. Thus, we arrive at
\begin{align*}
	\left\| 2^{ks} \phi_k(D)(f*g) \mid L_{\infty}(Q) \mid \ell_{\infty}(J;N)\right\|
	&\leq \|2^{ks} \phi_k(D) f \mid \ell_{\infty}(0;N) \mid L_1\|\cdot \|g \mid L_{\infty}\| \\
	&\leq \|f \mid F_{1,\infty}^s\| \cdot \|g \mid L_{\infty}\|.
\end{align*}
Using the monotone convergence theorem, we can let $N \to \infty$ and conclude that
\begin{equation*}
	\|f*g \mid F_{\infty,\infty}^s\|
	\leq \|f \mid F_{1,\infty}^s\| \cdot \|g \mid L_{\infty}\|.
\end{equation*}

\medskip
Case 4: $q=\infty$, $p_2= \infty$, $p_1<\infty$, and $A=B$. This follows with a similar -- and even simpler -- argument along the lines of Case 3, see \cite[Proof of Theorem 2]{burenkov90}.
\end{proof}

The second main result in this section is the following theorem.

\begin{theorem} \label{conv-3}
	Let $A \in \{B,F\}$ and $s,u \in \mbb{R}$. Let $q,q_1,q_2 \in (0,\infty]$ \textup{(}resp.\ $q,q_1,q_2 \in [1,\infty]$ for the $F$-scale\textup{)} and $p,p_1,p_2 \in [1,\infty]$ be such that
	\begin{equation}
		\frac{1}{q} \leq \frac{1}{q_1} + \frac{1}{q_2}
		\quad \text{and} \quad
		1+ \frac{1}{p} = \frac{1}{p_1}+\frac{1}{p_2}.
		\label{conv-eq21}
	\end{equation}
	If $f \in A_{p_1,q_1}^s$ and $g \in A_{p_2,q_2}^u$, then $f*g \in A_{p,q}^{s+u}$ and
	\begin{equation*}
		\|f*g \mid A_{p,q}^{s+u}\| \leq C \|f \mid A_{p_1,q_1}^s\| \cdot \|g \mid A_{p_2,q_2}^u\|
	\end{equation*}
    for some constant $C>0$, which depends only on the dyadic resolution $(\phi_k)_{k \in \mbb{N}_0}$ from the definition of $A_{p,q}^r$ and, additionally, on the dimension $n \in \mbb{N}$ if $\max\{p_1,p_2\}=\infty$.
\end{theorem}

The condition on the $p_i$ is again due to Young's inequality while the condition on the $q_i$ is needed in order to apply the H\"older inequality. For the particular case $q=q_1$ and $p=p_1$, the convolution estimate  reads
\begin{equation}
	\|f*g \mid A_{p,q}^{s+u}\| \leq C \|f \mid A_{p,q}^s\| \cdot \|g \mid A_{1,\infty}^{u}\|,
	\label{conv-eq23}
\end{equation}
and this estimate turns out to be a useful tool to study smoothing properties of convolution semigroups, see Section~\ref{semi}. For $u=0$ and $A=F$, we have
\begin{equation*}
	\|f*g \mid F_{p,q}^{s}\| \leq C \|f \mid F_{p,q}^s\| \cdot \|g \mid F_{1,\infty}^{0}\|,
\end{equation*}
which is an interesting complement to \eqref{conv-eq12} because $L_1 \not \subseteq F_{1,\infty}^0$. Theorem~\ref{conv-3} is known for the $B$-scale, cf.\ \cite[Theorem 3]{burenkov90}, but for the readers' convenience we also include the very short proof of this case.

\begin{proof}[Proof of Theorem~\ref{conv-3}]
    Since the norm $\|f \mid A_{p,q}^s\|$ is decreasing in $q$, we may assume that $\frac{1}{q} =\frac{1}{q_1}+\frac{1}{q_2}$. We will restrict our attention to the case $q,q_1,q_2<\infty$; if (at least) one of these numbers is infinite, the proof is similar and even becomes a bit simpler. Moreover, it follows as in the proof of Theorem~\ref{conv-1} that we may assume additionally that $f$ or $g$ is a Schwartz function; in either case, it holds that
	\begin{equation}
		\phi_k^2(D)(f*g) = (\phi_k(D)f)*(\phi_k(D)g) \label{conv-eq25}
	\end{equation}
    for the dyadic resolution $(\phi_k)_{k \in \mbb{N}_0}$ from the definition of the function spaces. Furthermore, we note that the sequence $(\phi_k^2)_{k \in \mbb{N}_0}$ is also a dyadic resolution, cf.\ \cite[proof of Theorem 3]{burenkov90} and the footnote on page~\pageref{blind-label}, and so
	\begin{align} \label{conv-eq26} \begin{aligned}
		\|h \mid B_{p,q}^{s+u}\| &\leq C \|2^{k(s+u)} \phi_k^2(D)h \mid L_p \mid \ell_q\|, \quad h \in B_{p,q}^{s+u} \\
		\|h \mid F_{p,q}^{s+u}\| &\leq C \|2^{k(s+u)} \phi_k^2(D) h \mid \ell_q \mid L_p\|, \quad h \in F_{p,q}^{s+u},\; p< \infty,
    \end{aligned}\end{align}
	for a uniform constant $C>0$.

	\primolist Case 1: $A=B$. By \eqref{conv-eq25}, \eqref{conv-eq26} and Young's inequality, we have
	\begin{align*}
		\|f*g \mid B_{p,q}^{s+u}\|
		&\leq C \|2^{k(s+u)} \phi_k^2(D)(f*g) \mid L_p \mid \ell_q\| \\
		&= C \|2^{k(s+u)} (\phi_k(D)f)*(\phi_k(D)g) \mid L_p \mid \ell_q\| \\
		&\leq C \big\| \|2^{ks} \phi_k(D)f \mid L_{p_1}\| \cdot \|2^{ku} \phi_k(D)g \mid L_{p_2}\| \mid \ell_q \big\|.
	\end{align*}
	Applying H\"older's inequality proves the assertion.

    \secundolist Case 2: $A=F$ and $p,p_1,p_2 < \infty$. From Minkowski's integral inequality \eqref{min-counting} and H\"older's inequality, we find for each $x \in \mbb{R}^n$ that
	\begin{align*}
		\|2^{k(s+u)} (\phi_k(D) f)*(\phi_k(D)g)(x) \mid \ell_q\|
		&\leq \int_{\mbb{R}^n} \|2^{k(s+u)} \phi_k(D)f(x-y) \cdot\phi_k(D)g(y) \mid \ell_q\| \, dy \\
		&\leq \int_{\mbb{R}^n} \|2^{ks} \phi_k(D) f(x-y) \mid \ell_{q_1}\| \cdot \|2^{ku} \phi_k(D) g(y) \mid \ell_{q_2}\| \, dy.
	\end{align*}
	Thus, by Young's inequality \eqref{young},
	\begin{align*}
		\|2^{k(s+u)} (\phi_k(D) f)*(\phi_k(D)g) \mid \ell_q \mid L_p\|
		&\leq \|2^{ks} \phi_k(D) f \mid \ell_{q_1} \mid L_{p_1}\| \cdot  \|2^{ku} \phi_k(D) g \mid \ell_{q_2} \mid L_{p_2}\| \\
		&= \|f \mid F_{p_1,q_1}^s\| \cdot \|g \mid F_{p_2,q_2}^u\|,
	\end{align*}
	which proves the assertion thanks to our earlier observations \eqref{conv-eq25}, \eqref{conv-eq26}.

    \tertiolist Case 3: $A=F$ and $\max\{p_1,p_2\}=\infty$. Since $f$ and $g$ play symmetric roles, it is enough to consider the case $p_1 = \infty$. Then, by \eqref{conv-eq21}, $p=\infty$ and $p_2=1$, i.e.\ we need to prove that
	\begin{align*}
		\|f*g \mid F_{\infty,q}^{s+u}\| \leq C \|f \mid F_{\infty,q_1}^s\| \cdot \|g \mid F_{1,q_2}^u\|.
	\end{align*}
	For every Borel set $Q \in \mc{B}(\mbb{R}^n)$ and every $J \geq 0$, Minkowski's integral inequality \eqref{min} shows that
	\begin{align*}
		&\left( \int_Q \sum_{k=J}^{\infty} |2^{k(s+u)} (\phi_k(D)f * \phi_k(D)g)(x)|^q \, dx \right)^{1/q} \\
	    &\quad\leq \int_{\mbb{R}^n} \left( \int_Q \sum_{k=J}^{\infty} |2^{k(s+u)} \phi_k(D) f(x-y) \cdot \phi_k(D) g(y)|^q \ dx \right)^{1/q} \, dy.
	\end{align*}
	H\"older's inequality gives the following upper bound for the integrand on the right-hand side:
	\begin{align*}
		&\left( \int_Q \sum_{k=J}^{\infty} |2^{k(s+u)} \phi_k(D) f(x-y) \cdot\phi_k(D) g(y)|^q \ dx \right)^{1/q} \\
        &\quad\leq \left( \int_Q \sum_{k=J}^{\infty} |2^{ks} \phi_k(D) f(x-y)|^{q_1} \, dx \right)^{1/q_1} \left( \int_Q \sum_{k=J}^{\infty} |\phi_k(D) g(y)|^{q_2} \, dx \right)^{1/q_2} \\
        &\quad\leq \sup_{z \in \mbb{R}^n} \left( \sum_{k=J}^{\infty} 2^{ksq} \int_{z+Q} |\phi_k(D) f(x)|^{q_1} \, dx \right)^{1/q_1} \lambda(Q)^{1/q_2} \|2^{ku} \phi_k(D) g(y) \mid \ell_{q_2}\|.
	\end{align*}
    Combining this with the previous estimate, dividing both sides by $\lambda(Q)^{1/q} = \lambda(Q)^{1/q_1} \lambda(Q)^{1/q_2}$ and using \eqref{conv-eq19}, we arrive at
	\begin{align*}
		&\left( \sum_{k=J}^{\infty} \sint_Q 2^{k(s+u)} \phi_k^2(D)(f*g)(x)|^q \, dx \right)^{1/q} \\
        &\quad\leq \|2^{ku} \phi_k(D) g \mid \ell_{q_2} \mid L_1\| \cdot \sup_{z \in \mbb{R}^n} \left( \sum_{k=J}^{\infty} 2^{ksq} \sint_{z+Q} |\phi_k(D) f(x)|^{q_1} \, dx \right)^{1/q_1} \\
		&\quad\leq \|g \mid F_{1,q_2}^u\| \cdot \|f \mid F_{\infty,q_1}^s\|
	\end{align*}
    for every $Q \in \mc{B}(\mbb{R}^n)$ with $\lambda(Q)<\infty$; hence, for every cube $Q_{J,M} := 2^{-J} M + (0,2^{-J})^n$. Since $(\phi_k^2)_{k \in \mbb{N}_0}$ is an admissible dyadic resolution, there is a uniform constant $C>0$ such that
	\begin{equation*}
		\|f*g \mid F_{\infty,q}^{s+u}\|
	    \leq C \sup_{J \geq 0, M \in \mbb{Z}^n} \left( \sum_{k=J}^{\infty} \sint_{Q_{J,M}} |2^{k(s+u)} \phi_k^2(D)(f*g)(x)|^q \, dx \right)^{1/q},
	\end{equation*}
	and the assertion follows.

    \quartolist Case 4: $A=F$ and $p=\infty$. From \eqref{conv-eq21}, we see that $p_2$ and $p_1$ are conjugate exponents, i.e.\ the assertion reads
	\begin{equation*}
		\|f*g \mid F_{\infty,q}^{s+u}\| \leq C \|f \mid F_{p_1,q_1}^s\| \cdot \|g \mid F_{p_1',q_2}^u\|.
	\end{equation*}
	We have already shown in the previous step that this inequality holds if $p_1 = \infty$, and so it remains to consider the case $p_1<\infty$. Applying H\"older's inequality twice and invoking Minkowski's integral inequality \eqref{min-counting}, we get
	\begin{align*}
		\left( \sum_{k=J}^{\infty} |2^{k(s+u)} (\phi_k(D) f*\phi_k(D) g)(x)|^q \right)^{1/q}
		&\leq \|2^{k(s+u)} \phi_k(D) f(x-\bullet) \phi_k(D) \cdot g(\bullet) \mid \ell_q \mid L_1\| \\
        &\leq \big\| \|2^{ks} \phi_k(D) f(x-\bullet) \mid \ell_{q_1} \| \cdot \|2^{ku} \phi_k(D) g(\bullet) \mid \ell_{q_2}\|  \mid L_1\big\| \\
		&\leq \|2^{ks} \phi_k(D) f \mid \ell_{q_1} \mid L_{p_1}\| \cdot \|2^{ku} \phi_k(D) g\mid \ell_{q_2} \mid L_{p_1'}\| \\
		&= \|f \mid F_{p_1,q_1}^s\| \cdot \|g \mid F_{p_1',q_2}^u\|
	\end{align*}
	for all $x \in \mbb{R}^n$. Thus,
	\begin{align*}
		\left( \sum_{k=J}^{\infty} \sint_Q |2^{k(s+u)} \phi_k^2(D)(f*g)(x)|^q \, dx \right)^{1/q}
		&= \left( \sint_Q \sum_{k=J}^{\infty} |2^{k(s+u)} (\phi_k(D)f * \phi_k(D) g)(x)|^q \, dx \right)^{1/q} \\
		&\leq \|f \mid F_{p_1,q_1}^s\| \cdot \|g \mid F_{p_1',q_2}^u\|
	\end{align*}
	for every Borel set $Q \in \mc{B}(\mbb{R}^n)$ with $\lambda(Q)<\infty$. By our earlier considerations, this proves the assertion, compare the conclusion of Case 3.
\end{proof}

\section{Regularity of convolution semigroups} \label{semi}

The convolution inequalities from the previous section are useful to study mapping properties of convolution semigroups, considered as operators on the function spaces $A_{p,q}^s$, $A \in \{B,F\}$. We start with some basic material on convolution semigroups.

\begin{definition} \label{semi-1}
    Let $p \in [1,\infty]$, and let $P_t: L_p \to L_p$, $t \geq 0$, be a family of linear operators. If there exists for every $t \geq 0$ some (possibly signed) measure $\mu_t$ such that $P_t f = f*\mu_t$, then $(P_t)_{t \geq 0}$ is a \emph{family of convolution operators}. If, additionally, $P_0=\id$ and $P_{t+s} = P_t P_s$ for all $s,t \geq 0$, then $(P_t)_{t \geq 0}$ is a \emph{convolution semigroup}.
\end{definition}

Here, we restrict our attention to convolution operators $P_t f = f*\mu_t$ with $\mu_t(dy) = p_t(y) \, dy$ for $p_t \in L_1$. Note that the kernels $p_t$ need not be non-negative. If $P_t$ is a convolution operator with kernel $p_t$, then, by Young's inequality \eqref{young},
\begin{equation*}
	\|P_tf \mid L_p\| \leq \|p_t \mid L_1\| \cdot \|f \mid L_p\|, \quad p \in [1,\infty],
\end{equation*}
and so $p_t \in L_1$ implies that $P_t:L_p\to L_p$ is a bounded operator\footnote{In general, integrability of the kernel is not necessary for a convolution operator to be bounded. For instance, if $p=2$, then $Pf:=f*g$ with $g(x):=\frac{1}{x} \sin x$ defines, by Plancherel's theorem, a bounded operator on $L_2$ but $g \notin L_1$.}. The convolution inequalities from Section~\ref{conv} allow us to study the mapping properties of the convolution operator $P_t$ acting on the function spaces $A_{p,q}^s$. By Theorem~\ref{conv-1} and \ref{conv-3}, we have
\begin{align}
	\|P_t f \mid A_{p,q}^s\| \leq C \|p_t \mid L_1\| \cdot \|f \mid A_{p,q}^s\|
	\quad \text{and} \quad
	\|P_t f \mid A_{p,q}^{s+u}\| \leq C' \|p_t \mid A_{1,\infty}^u\| \cdot \|f \mid A_{p,q}^{s}\|,
	\label{semi-eq5}
\end{align}
which means that $p_t \in L_1$ (resp.\ $p_t \in A_{1,\infty}^u$) implies that $P_t$ defines a bounded linear operator on $A_{p,q}^s$ (resp.\ from $A_{p,q}^s$ to $A_{p,q}^{s+u}$). Consequently, the main object of interest are the norms $\|p_t \mid A_{p,q}^s\|$ -- having estimates for these norms at hand, the convolution inequalities from Section~\ref{conv} yield immediately bounds for $\|P_t f \mid A_{\tilde{p},\tilde{q}}^{\tilde{s}}\|$. \par
The first aim of this section is to show that the kernels $p_t$, $t>0$, of a convolution semigroup are self-regularizing in time, in the sense that the regularity which has been reached at some time $T>0$ is preserved for all larger times $t>T$. More precisely, we will see that $p_T \in A_{p,q}^s$ for some $T>0$ implies that $p_t \in A_{p,q}^s$ for all $t>T$ provided that $p_t$ is integrable for $t>0$. The proof is based on the fact that the kernels of a convolution semigroup satisfy the so-called \emph{Chapman--Kolmogorov equations}; for the readers' convenience we state the result and give the proof.

\begin{lemma} \label{semi-3}
	Let $(p_t)_{t>0} \subseteq L_1$. Denote by $P_t f:=f*p_t$, $t>0$, the corresponding family of convolution operators and set $P_0 = \id$. The family $(P_t)_{t \geq 0}$ is a convolution semigroup if, and only if, $(p_t)_{t>0}$ satisfies the Chapman--Kolmogorov equations
	\begin{equation}
		p_{t+s}(x) = p_s * p_t(x) = \int_{\mbb{R}^n} p_t(y) p_s(x-y) \, dy \quad \text{a.e.} \label{semi-eq7}
	\end{equation}
	for all $s,t>0$; the exceptional Lebesgue null set may depend on $s$, $t$.
\end{lemma}

\begin{proof}
	If $(p_t)_{t>0}$ satisfies the Chapman--Kolmogorov equations \eqref{semi-eq7}, then an application of Fubini's theorem shows that $P_{t+r} f = P_t P_r f$ a.e.\ for every bounded measurable function $f \in B_b(\mbb{R}^n)$ and any $t,r \geq 0$. By Young's inequality, each $P_r$, $r \geq 0$, is a bounded linear operator on $L_p$, and  so the identity $P_{t+r}f=P_t P_r f$ extends from $f \in L_p \cap B_b(\mbb{R}^n)$ to $f \in L_p$, $p \in [1,\infty]$, using a standard approximation argument.  \par
	Conversely, if $(P_t)_{t \geq 0}$ satisfies the semigroup property, then Fubini's theorem gives $p_{t+r}*f = q_{r,t}*f$ a.e. for $f \in C_c^{\infty}(\mbb{R}^n)$, where
	\begin{equation*}
		q_{r,t}(x) := \int_{\mbb{R}^n} p_r(x-y) p_t(y) \, dy = p_r*p_t(x).
	\end{equation*}
	 Noting that $p_{t+r}*f$ and $q_{r,t}*f$ are continuous, we see that $(p_{t+r}*f)(0)= (q_{r,t}*f)(0)$, i.e.\
	\begin{equation*}
		\int_{\mbb{R}^n} p_{t+r}(y) f(-y) \, dy
		= \int_{\mbb{R}^n} q_{r,t}(y) f(-y) \, dy
	\end{equation*}
	for all $f \in C_c^{\infty}(\mbb{R}^n)$, and this implies $p_{t+r}=q_{r,t}$ a.e.
\end{proof}

\begin{corollary} \label{semi-5}
    Let $(p_t)_{t>0} \subseteq L_1$ be the kernels of a convolution semigroup $(P_t)_{t \geq 0}$. Let $A \in \{B,F\}$, $p \in [1,\infty]$, $q  \in (0,\infty]$ \textup{(}$q \geq 1$ if $A=F$\textup{)} and $s \in \mbb{R}$. If $p_T \in A_{p,q}^s$ for some $T>0$, then $p_t \in A_{p,q}^s$ for all $t \geq T$ and
	\begin{equation*}
		\|p_t \mid A_{p,q}^s\| \leq \|p_{t-T} \mid L_1\| \cdot \|p_T \mid A_{p,q}^s\|, \quad t >T.
	\end{equation*}
\end{corollary}

\begin{proof}
	By Lemma~\ref{semi-3}, we have $p_t = p_T*p_{t-T}$ for all $t>T$, and the assertion is immediate from Theorem~\ref{conv-1}.
\end{proof}

Let us mention that there is no self-regularizing property backwards in time, i.e.\ $p_T \in A_{p,q}^s$ does not imply $p_t \in A_{p,q}^s$ for $t<T$. For instance, it is possible to construct a  convolution semigroup with (positive) kernels $(p_t)_{t>0} \subseteq L_1$ such that $p_t$ is discontinuous for $t<1$ while there is for every $k \in \mbb{N}$ some $T=T(k)>0$ such that $p_t \in C^k$ for $t \geq T(k)$, cf.\ \cite[Example 23.3]{sato} and \cite[Example 23]{knop13}. \par \medskip

For certain kernels $(p_t)_{t>0}$ of convolution semigroups it is known that integrals of higher-order derivatives can be controlled by integrals of the gradient, i.e.\ one has estimates of the form
\begin{equation*}
    \int_{\mbb{R}^n} |\partial^{\beta} p_t(x)| \,dx \leq C \left( \int_{\mbb{R}^n} |\nabla p_t(x)| \, dx \right)^{|\beta|},
    \quad \beta \in \mbb{N}_0^n, \; |\beta| := \sum_{i=1}^n \beta_i,
\end{equation*}
and $p_t \in C^1$, $\nabla p_t \in L_1$ for $t \in (0,T)$, implies $p_t \in C^k$ for all $t>0$; this was used e.g.\ in \cite{reg-levy,euler} to study the regularity of functions related to L\'evy semigroups. Our next result and Theorem~\ref{semi-11} below establish analogous statements for Besov spaces and Triebel--Lizorkin spaces.

\begin{corollary} \label{semi-7}
	Let $(p_t)_{t>0} \subseteq L_1$ be the kernels of a convolution semigroup $(P_t)_{t \geq 0}$. Let $A \in \{B,F\}$, $p \in [1,\infty]$, $q \in (0,\infty]$ \textup{(}$q \geq 1$ if $A=F$\textup{)} and $s \in \mbb{R}$. If there is some $T>0$ such that $p_t \in A_{p,q}^s \cap A_{1,\infty}^s$ for all $t \in (0,T]$, then $p_t \in A_{p,q}^{sk}$ for all $k \in \mbb{N}$ and all $t>0$, and
	\begin{equation}
		\|p_t \mid A_{p,q}^{sk}\|
		\leq C^k \begin{cases}
			\|p_{t/k} \mid A_{1,\infty}^s\|^{k-1} \cdot \|p_{t/k} \mid A_{p,q}^s\|, & \text{if $t \in (0,T]$}, \\
			\|p_{T/k} \mid A_{1,\infty}^s\|^{k-1} \cdot \|p_{T/k} \mid A_{p,q}^s\| \cdot \|p_{t-T} \mid L_1\|, & \text{if $t>T$},
		\end{cases} \label{semi-eq11}
	\end{equation}
	for some uniform constant $C \geq 1$.
\end{corollary}

\begin{proof}
	Fix $t \in (0,T]$, then by the Chapman--Kolmogorov equation, cf.\ Lemma~\ref{semi-3},
	\begin{equation*}
		p_t = \underbrace{p_{t/k}* \ldots * p_{t/k}}_{\text{$k$ times}},
	\end{equation*}
    and the estimate for $\|p_t \mid A_{p,q}^{sk}\|$ follows by repeated applications of Theorem~\ref{conv-3}. For $t>T$, the estimate is a direct consequence of Corollary~\ref{semi-5}.
\end{proof}

Corollary~\ref{semi-7} shows, in particular, that $p_t \in A_{1,\infty}^s$, $t \in (0,T]$, implies $p_t \in A_{1,\infty}^s$ for all $t>0$ and
	\begin{equation}
		\|p_t \mid A_{1,\infty}^{sk}\|
		\leq C^k \begin{cases}
			\|p_{t/k} \mid A_{1,\infty}^s\|^{k}, & \text{if $t \in (0,T]$}, \\
			\|p_{T/k} \mid A_{1,\infty}^s\|^{k} \cdot \|p_{t-T} \mid L_1\|, & \text{if $t>T$}.
		\end{cases} \label{semi-eq13}
\end{equation}

\begin{remark} \label{semi-8}
It is possible to extend the estimates from \eqref{semi-eq11} and \eqref{semi-eq13} using interpolation. Indeed, by Corollary~\ref{semi-7}, we have
\begin{equation*}
	\|p_t \mid A_{p,q}^{sk}\| \leq C^k\kern-7pt\sup_{\frac{t}{k+1} \leq r \leq \frac{t}{k}} \left( \|p_r \mid A_{1,\infty}^s\|^{k-1} \|p_r \mid A_{p,q}^s\| \right), \quad k \in \mbb{N},\;t \in (0,T],
\end{equation*}
and from the Riesz--Thorin interpolation theorem we find that
\begin{equation*}
	\|p_t \mid A_{p,q}^{s\varrho}\| \leq C^\varrho\kern-9pt\sup_{\frac{t}{\floor{\varrho}+1} \leq r \leq \frac{t}{\floor{\varrho}}} \left( \|p_r \mid A_{1,\infty}^s\|^{\varrho-1} \|p_r \mid A_{p,q}^s\| \right), \quad t \in (0,T],
\end{equation*}
for any $\varrho \geq 1$, i.e.\
\begin{equation}
	\|p_t \mid A_{p,q}^{u}\| \leq C^{u/s}\kern-9pt\sup_{\frac{t}{\floor{u/s}+1} \leq r \leq \frac{t}{\floor{u/s}}} \left( \|p_r \mid A_{1,\infty}^s\|^{u/s-1} \|p_r \mid A_{p,q}^s\| \right), \quad t \in (0,T], \label{semi-eq15}
\end{equation}
for all $u \geq s$. Note that, by Corollary~\ref{semi-5}, this entails automatically estimates for $t>T$.
\end{remark}

In order to apply Corollary~\ref{semi-7}, one needs to bound $\|p_t \mid A_{1,\infty}^s\|$ for (at least) small $t>0$ and some $s>0$. Sometimes it is convenient to work with other function spaces to obtain such estimates, e.g.\
\begin{itemize}
	\item the Bessel potential space $H_1^s$ with norm
	\begin{equation}
		\|f \mid H_1^s\| := \left\| \mc{F}^{-1}((1+|\bullet|^2)^{s/2} \mc{F}f) \mid L_1 \right\|, \quad s \geq 0. \label{semi-eq18}
	\end{equation}
	\item the Sobolev space $W_1^m$ with norm
	\begin{equation*}
		\|f \mid W_1^m\| := \left\| \sum\nolimits_{|\alpha| \leq m} |\partial^{\alpha} f| \mid L_1 \right\|, \quad m \in \mbb{N}_0.
	\end{equation*}
	\item the local Hardy space $h_1$ with norm \begin{equation*}
		\|f \mid h_1\| = \left\| \sup\nolimits_{t \in (0,1)} |\phi(tD)f| \mid L_1 \right\|
	\end{equation*}
	for some fixed $\phi \in \mc{S}(\mbb{R}^n)$ with $(\mc{F}\phi)(0)=1$, cf.\ \cite[Section 1.4.4]{triebel2}.
\end{itemize}
In the next result we collect some inequalities, which can be useful to estimate $\|p_t \mid A_{1,\infty}^s\|$.

\begin{proposition} \label{semi-9}
Let $s \in \mbb{R}$, $A \in \{B,F\}$. The following inequalities hold for suitable constants $C>0$ \textup{(}not depending on $f$\textup{)}.
\begin{enumerate}\renewcommand{\itemsep}{4pt}
    \item\label{semi-9-i} $\|f \mid A_{1,\infty}^s\| \leq C \sup_{0 \leq |\alpha| \leq m} \|\partial^{\alpha} f \mid A_{1,\infty}^{s-m}\|$ for all $m \in \mbb{N}$.
	
    \item\label{semi-9-ii} $\|f \mid A_{1,\infty}^s\| \leq C \|f \mid B_{1,1}^s\|= C\|f \mid F_{1,1}^s\|$.
	
    \item\label{semi-9-iii} $\|f \mid A_{1,\infty}^s\| \leq C \|f \mid L_1\|^{1-s/m} \cdot \|f \mid \tilde{A}_1^{m}\|^{s/m}$ for $m \in \mbb{N}$ even, $s \in (0,m)$ and $\tilde{A} \in \{W,H\}$.
	
    \item\label{semi-9-iv} $\|f \mid F_{1,\infty}^0\| \leq C \|f \mid h_1\|$. In particular,
	\begin{equation*}
		\|f \mid F_{1,\infty}^0\| \leq C \left\| \sup\nolimits_{t \in (0,1)} |q_t*f| \mid L_1 \right\|
	\end{equation*}
	for the Gaussian heat kernel $q_t$ \textup{(}i.e.\ the transition density of Brownian motion\textup{)}.
	
    \item\label{semi-9-v} $\|f \mid B_{1,\infty}^s\| \leq C \|f \mid H_1^s\|$. In particular,  $\|f \mid B_{1,\infty}^0\| \leq C \|f \mid L_1\|$.
\end{enumerate}
\end{proposition}

\begin{proof}
\begin{enumerate}[wide, labelwidth=!, labelindent=0pt]\renewcommand{\itemsep}{6pt}
\item See \cite[Theorem 1.24]{triebel4}.

\item The inequality $\|f \mid A_{1,\infty}^s\| \leq C \|f \mid A_{1,1}^s\|$ is immediate from the fact that $A_{1,1}^s$ is continuously embedded into $A_{1,\infty}^s$, see e.g.\ \cite[Theorem 2.9]{triebel4}. Moreover, $F_{1,1}^s = B_{1,1}^s$ by definition.

\item The Sobolev space $W_1^{m}$ is continuously embedded into $H_1^{m}$ for $m \in \mbb{N}$ even\footnote{This breaks down for odd integers $m$, cf.\ \cite[p.~160, Problem 6.6]{stein}}, and so it is enough to prove the inequality for $\tilde{A}=H$. Since $B_{1,1}^s$ is a real interpolation space of $H_1^0 = L_1$ and $H_1^{m}$, cf.\ \cite[Theorem 6.2.4, No.\ (10)]{bergh}, the Riesz--Thorin theorem shows that
	\begin{equation*}
		\|f \mid B_{1,1}^s\| \leq C \|f \mid L_1\|^{1-s/m} \cdot \|f \mid H_1^{m}\|^{s/m}, \quad s \in (0,m).
	\end{equation*}
	By \ref{semi-9-ii}, this proves the inequality as claimed.

\item The first inequality follows from the fact that $h_1 = F_{1,2}^0$ is continuously embedded into $F_{1,\infty}^0$,
    cf.\ \cite[Theorem 3.18]{sawano} or \cite[2.3.5(4)]{triebel1}. For the second estimate, choose $\phi(x)=e^{-|x|^2}$ in the definition of the norm on the local Hardy space.

\item See \cite[Theorem 6.2.4, No.\ (9)]{bergh} for the first assertion; the second one is then immediate if we take $s=0$.	 \qedhere
\end{enumerate}
\end{proof}

If the kernels $p_t$ of a convolution semigroup are differentiable and $\nabla p_t \in L_1$, then $\|p_t \mid A_{1,\infty}^u\|$ can be estimated from above by a power of the $L^1$-norm of the gradient $\nabla p_t$. This result is very useful because it is possible in many cases to estimate $\int |\nabla p_t(x)| \,dx$, see below for examples.  Note that $\int |\nabla p_t(x)| \, dx$ is the operator norm of the operator $\nabla P_t: B_b(\mbb{R}^n) \to B_b(\mbb{R}^n)$, cf.\ \cite[Lemma 4.1]{euler}.

\begin{theorem} \label{semi-11}
    Let $(P_t)_{t \geq 0}$ be a convolution semigroup with kernels $(p_t)_{t>0} \subseteq L_1$. If $p_t \in C^1$ and $\partial_{x_i} p_t \in L_1$ for all $i \in \{1,\ldots,n\}$, $t>0$, then
	\begin{equation}\label{semi-eq21}
		\|p_t \mid A_{1,\infty}^u\|
        \leq C_u \kern-9pt \sup_{\frac{t}{\floor{u}+1} \leq r \leq \frac{t}{\max\{\floor{u},1\}}} \left( \|p_r \mid L_1\| + \|\nabla p_{r/2} \mid L_1\| \right)^u
	\end{equation}
    for every $t>0$, $u>0$, $A \in \{B,F\}$ and some constant $C_u>0$. In particular,  $P_t: A_{p,q}^s \to A_{p,q}^{s+u}$ is a bounded linear operator for any $t>0$, $p,q \in [1,\infty]$, $s \in \mbb{R}$, $u \geq 0$,  and
	\begin{equation*}
		\|P_t f \mid A_{p,q}^{s+u}\|
        \leq C_u \kern-9pt \sup_{\frac{t}{\floor{u}+1} \leq r \leq \frac{t}{\max\{\floor{u},1\}}} \left( \|p_r \mid L_1\| + \|\nabla p_{r/2} \mid L_1\| \right)^u \|f \mid A_{p,q}^s\|.
	\end{equation*}
\end{theorem}

\begin{proof}
	The semigroup property entails that
	\begin{equation}
		\int_{\mbb{R}^n} |\partial_{x_i}^2 p_t(x)| \, dx
		\leq \left( \int_{\mbb{R}^n} |\partial_{x_i} p_{t/2}(x)| \, dx \right)^2 \label{semi-eq23}
	\end{equation}
	for every $i \in \{1,\ldots,n\}$. This was essentially shown in \cite[Lemma 4.1]{euler}, see also \cite[Proposition 3.1]{reg-levy}; let us briefly explain the proof. By the Chapman--Kolmogorov equation, cf.\ Lemma~\ref{semi-3},
	\begin{equation*}
		p_{t}(x) = \int_{\mbb{R}^n} p_{t/2}(x-y) p_{t/2}(y) \, dy,
	\end{equation*}
	and so the differentiation lemma for parametrized integrals yields
	\begin{equation*}
		\partial_{x_i} p_t(x)
		= \int_{\mbb{R}^n} \partial_{x_i} p_{t/2}(x-y) p_{t/2}(y) \, dy
		= \int_{\mbb{R}^n} \partial_{x_i} p_{t/2}(z) p_{t/2}(x-z) \, dz
	\end{equation*}
	and
	\begin{equation*}
		\partial_{x_i}^2 p_{t}(x)
		= \int_{\mbb{R}^n} \partial_{x_i} p_{t/2}(z) \partial_{x_i} p_{t/2}(x-z) \, dz.
	\end{equation*}
	Applying Tonelli's theorem gives \eqref{semi-eq23}. By \eqref{semi-eq23},
	\begin{equation*}
		\|p_t \mid H_1^2\| \leq \|p_t \mid L_1\| + n \left( \int_{\mbb{R}^n} |\nabla p_{t/2}(x)| \, dx \right)^2,
	\end{equation*}
	where $H_1^2=H_1^2(\mbb{R}^n)$ denotes the Bessel potential space, cf.\ \eqref{semi-eq18}. Consequently, Proposition~\ref{semi-9}\ref{semi-9-iii} entails for every $s \in (0,2)$ that
	\begin{align*}
		\|p_t \mid A_{1,\infty}^s\| \leq C \left(\|p_t \mid L_1\| + \int_{\mbb{R}^n} |\nabla p_{t/2}(x)| \, dx \right)^{s}
	\end{align*}
    for some constant $C=C(s,n)$. In particular, \eqref{semi-eq21} holds for $u \in (0,2)$. Since Corollary~\ref{semi-7} and Remark~\ref{semi-8} show that
	\begin{equation*}
		\|p_t \mid A_{1,\infty}^u\|
		\leq \tilde{C}^u \sup_{\frac{t}{\floor{u}+1} \leq r \leq \frac{t}{\floor{u}}} \|p_r \mid A_{1,\infty}^1\|^u, \quad t>0,\; u \geq 1,
	\end{equation*}
	for a uniform constant $\tilde{C} \geq 1$, we obtain that
	\begin{equation*}
		\|p_t \mid A_{1,\infty}^u\|
		\leq C^u \tilde{C}^u \sup_{\frac{t}{\floor{u}+1} \leq r \leq \frac{t}{\floor{u}}} \left( \|p_r \mid L_1\| + \| \nabla p_{r/2} \mid L_1\| \right)^u, \quad t>0,\; u \geq 1,
	\end{equation*}
	which finishes the proof of \eqref{semi-eq21}.	The second assertion is immediate from \eqref{semi-eq5}.
\end{proof}

To illustrate the results from this section, we apply them to the \emph{generalized Gau{\ss}--Weierstra{\ss} semigroup}, i.e.\ the family of convolution operators $(P_t^{(m)})_{t \geq 0}$, $m \in \mbb{N}$, defined by
\begin{equation}
	\mc{F}(P_t^{(m)} u)(\xi) = e^{-t |\xi|^{2m}} \mc{F}u(\xi), \quad u \in \mc{S}(\mbb{R}^n),\; \xi \in \mbb{R}^n, \label{semi-eq27}
\end{equation}
If $m=1$, then this is the classical Gau{\ss}--Weierstra{\ss} semigroup, which can understood as transition semigroup of Brownian motion, see also Section~\ref{levy}. For $m>1$, we still have that $P_t^{(m)}$ has a rapidly decaying, smooth kernel $p_t^{(m)}$,
\begin{equation*}
	p_t^{(m)}(x) :=\frac{1}{(2\pi)^{n/2}} \mc{F}^{-1}(e^{-t |\bullet|^{2m}})(x), \quad x \in \mbb{R}^n;
\end{equation*}
however, $p_t^{(m)}$ is not non-negative for $m>1$, and therefore $(P_t^{(m)})_{t \geq 0}$ cannot be interpreted as a transition semigroup of a stochastic process. The generalized Gau{\ss}--Weierstra{\ss} semigroup is closely related to the  polyharmonic heat equation $(-\Delta)^m u(t,x)=\partial_t u(t,x)$, see e.g. \cite{baaske,barbatis,davies95} and the references therein.

\begin{corollary}\label{semi-13}
	Let $m \in \mbb{N}$, $A \in \{B,F\}$, $p,q \in [1,\infty]$, $t>0$, $s \in \mbb{R}$ and $u \geq 0$. Then \begin{equation*}
		 \|P_t^{(m)} f \mid A_{p,q}^{s+u}\| \leq C t^{-u/2m} \|f \mid A_{p,q}^s\|.
	\end{equation*}
\end{corollary}

For $m=1$, Corollary~\ref{semi-13} yields the classical estimates for the Gau{\ss}--Weierstra{\ss} semigroup, cf.\ \cite[Theorem 3.35]{triebel4}. For $m>1$, our estimates match the ones obtained by Baaske \& Schmei{\ss}er \cite[Theorem 3.5]{baaske}; they use a different method which does not cover the case $A=F$ and $p=\infty$.

\begin{proof}[Proof of Corollary~\ref{semi-13}]
	Fix $m \in \mbb{N}$ and denote by $p_t^{(m)}$ the kernel of $P_t^{(m)}$, \begin{equation*}
		p_t(x) := p_t^{(m)}(x) := \frac{1}{(2\pi)^{n}} \int_{\mbb{R}^n} e^{ix \cdot \xi} e^{-t |\xi|^{2m}} \, d\xi, \quad x \in \mbb{R}^n.
	\end{equation*}
	By the substitution rule, $\eta=t^{1/2m} \xi$, $d\eta = t^{n/2m} d\xi$, we get the following scaling relation: \begin{equation*}
		p_t(x)
		= t^{-n/2m} \frac{1}{(2\pi)^{n}} \int_{\mbb{R}^n} e^{it^{-1/2m} x \eta} e^{-|\eta|^{2m}} \, d\xi
		= t^{-n/2m} p_1(t^{-1/2m} x).
	\end{equation*}
	In particular,
    \begin{equation*}
		\int_{\mbb{R}^n} |p_t(x)| \, dx
		= t^{-n/2m} \int_{\mbb{R}^n} |p_1(t^{-1/2m} x)| \, dx
		= \int_{\mbb{R}^n} |p_1(x)| \, dx < \infty
	\end{equation*}
	for all $t>0$, i.e.\ $\|p_t \mid L_1\| = \|p_1 \mid L_1\|$ for all $t>0$. Analogously, we find from
    \begin{equation*}
		\partial_{x_j} p_t^{(m)}(x)
		= t^{-(n+1)/2m} (\partial_{x_j} p_1)(t^{-1/2m} x)
	\end{equation*}
	that
	\begin{equation*}
		\int_{\mbb{R}^n} |\nabla p_t^{(m)}(x)| \, dx
		\leq C t^{-1/2m}
	\end{equation*}
	for some uniform constant $C>0$. Applying Theorem~\ref{semi-11} finishes the proof.
\end{proof}

\begin{remark}
    The family of operators $(P_t^{(m)})_{t \geq 0}$ from \eqref{semi-eq27} defines a strongly continuous and conservative convolution semigroup for every $m\in (0,\infty)$, i.e.\ $m$ need not be an integer. The semigroup  $(P_t^{(m)})_{t \geq 0}$ satisfies the convolution estimates from Corollary~\ref{semi-13} for \emph{any} $m \in (0,\infty)$. This can be deduced from the convolution estimates for integer $m \in \mbb{N}$ using Bochner's subordination, see \cite{knop19}; at the end of Section~\ref{levy} we give another, direct proof of this result. Let us to mention that our convolution estimates can be used, more generally, to study the smoothing properties of higher order semigroups; more precisely, for semigroups satisfying
	\begin{equation*}
		\mc{F}(P_t^{(m)} f) = e^{-t\psi^m} \mc{F}f,
	\end{equation*}
    for a continuous negative definite function $\psi$, cf.\ Section~\ref{levy}.  Corollary~\ref{semi-13} may be seen as a special case where $\psi(\xi) = |\xi|^2$.
\end{remark}

\section{Regularity of Markovian convolution semigroups} \label{levy}

In the previous section, we studied convolution semigroups $P_t f = f *\mu_t$ for (not necessarily positive) measures $\mu_t$. Now we focus on the smaller class of convolution semigroups, that is, we now assume that the measures $\mu_t$ on $(\mbb{R}^n,\mc{B}(\mbb{R}^n))$ are sub-probability measures: i.e.\ non-negative measures satisfying $\mu_t(\mbb{R}^n) \leq 1$. In particular, the semigroup $(P_t)_{t \geq 0}$ is positivity preserving and sub-Markovian, that is,
\begin{equation*}
	0 \leq f \leq 1 \implies 0 \leq P_t f \leq 1.
\end{equation*}
We call $(P_t)_{t \geq 0}$ a \emph{Markovian convolution semigroup}. Such semigroups appear naturally in probability theory in connection with L\'evy processes \cite{jacob1,barcelona,sato}. The measure $\mu_t$ can be interpreted as the distribution $\mbb{P}(X_t \in \cdot)$ of a \emph{L\'evy process} $(X_t)_{t \geq 0}$, i.e.\ a stochastic process with independent and stationary increments which is continuous in probability. In fact, there is a one-to-one correspondence between L\'evy processes and Markovian convolution semigroups. We will continue using the analysts' language, but it is useful to keep in mind that our results do have an interpretation -- as well as applications -- in probability theory.

If $(P_t)_{t \geq 0}$ is a Markovian convolution semigroup such that $\mu_t(dy)=p_t(y) \, dy$ is absolutely continuous with respect to Lebesgue measure, then it follows immediately from Section~\ref{conv} that
\begin{equation*}
	\|P_t f \mid A_{p,q}^{s}\| \leq \|f \mid A_{p,q}^s\|
	\quad \text{and} \quad
	\|P_t f \mid A_{p,q}^{s+u}\| \leq C \|p_t \mid A_{1,\infty}^u\| \cdot \|f \mid A_{p,q}^s\|,
\end{equation*}
see also Theorem~\ref{levy-3} below. In this section, we will show how to bound $\|p_t \mid A_{1,\infty}^u\|$; we will present some general estimates and tools which allow us to bound the norm for a wide class of Markovian convolution semigroups. Our analysis relies on the fact that Markovian convolution semigroups can be completely characterized using the Fourier transform. One has, cf.\ \cite[Section 3.6]{jacob1},
\begin{equation}
	\mc{F}(P_t f)(\xi) = e^{-t \psi(\xi)} \mc{F}(f)(\xi), \label{levy-eq3}
\end{equation}
for a function $\psi: \mbb{R}^n \to \mbb{C}$ which is continuous and negative definite (in the sense of Schoenberg); we call $\psi$ the \emph{characteristic exponent} of $(P_t)_{t \geq 0}$. Any such $\psi$ is uniquely characterized by its L\'evy--Khintchine representation
\begin{equation}
	\psi(\xi)= a -i b \cdot \xi + \frac{1}{2} \xi \cdot Q \xi + \int_{y \neq 0} \left(1-e^{iy \cdot \xi}+iy \cdot \xi \I_{(0,1)}(|y|) \right)  \nu(dy), \label{levy-eq5}
\end{equation}
where $a \geq 0$, $b \in \mbb{R}^n$, $Q \in \mbb{R}^{n \times n}$ is positive semidefinite and $\nu$ is a measure on $\mbb{R}^n \setminus \{0\}$ such that $\int_{y \neq 0} \min\{1,|y|^2\} \, \nu(dy)<\infty$. Important examples include $\psi(\xi)=|\xi|^2$ (Gau{\ss}--Weierstra{\ss} semigroup), $\psi(\xi)=|\xi|$ (Cauchy--Poisson semigroup), $\psi(\xi)=|\xi|^{\alpha}$, $\alpha \in (0,2)$ (stable semigroup), and $\psi(\xi)=\log(1+|\xi|)$ (Gamma semigroup), just to mention a few.

Our starting point is the following theorem, which is immediate from Section~\ref{conv}.

\begin{theorem} \label{levy-3}
    Let $(P_t)_{t \geq 0}$ be a Markovian convolution semigroup. If $t>0$ is such that $P_t$ has kernel $p_t$, i.e.\ $\mu_t(dy)=p_t(y) \, dy$, then the following assertions hold for any choice of $A \in \{B,F\}$, $s,u \in \mbb{R}$ and $p,q \in [1,\infty]$:
	\begin{enumerate}\renewcommand{\itemsep}{4pt}
    \item\label{levy-3-i} $P_t: A_{p,q}^s \to A_{p,q}^s$ is a linear contraction operator if $p<\infty$.

    \item\label{levy-3-ii} $\|P_t f \mid A_{p,q}^{s+u}\| \leq C \|p_t \mid A_{1,\infty}^u\| \cdot \|f \mid A_{p,q}^s\|$; here $C>0$ is a uniform constant.
	\end{enumerate}
\end{theorem}

\begin{remark} \label{levy-5}
    There is a necessary condition in terms of the characteristic exponent $\psi$ of $(P_t)_{t \geq 0}$ ensuring the existence of the kernel $p_t$. If $\psi$ satisfies the so-called \emph{Hartman--Wintner condition}
	\begin{equation}
		\lim_{|\xi| \to \infty} \frac{\re \psi(\xi)}{\log |\xi|} = \infty, \label{levy-eq11}
	\end{equation}
	then there is for every $t>0$ a non-negative function $p_t \in L_1 \cap C_b^{\infty}(\mbb{R}^n)$ such that $P_t f = f*p_t$, see \cite{knop13} for a detailed discussion. Under the milder assumption that there is some $M>0$ such that
	\begin{equation*}
		\lim_{|\xi| \to \infty} \frac{\re \psi(\xi)}{\log |\xi|} \geq M,
	\end{equation*}
	the operator $P_t$ has a kernel $p_t$ for sufficiently large $t \gg1$. These growth conditions are satisfied for a wide class of Markovian convolution semigroups. From the point of view of probability theory, the kernel $p_t$ is the transition density of the L\'evy process $(X_t)_{t \geq 0}$ associated with the semigroup $(P_t)_{t \geq 0}$, i.e.\ $\mbb{P}(X_t \in B) = \int_B p_t(y) \, dy$ for any Borel set $B \in \mc{B}(\mbb{R}^n)$.
\end{remark}

\begin{corollary} \label{levy-7}
	Let $(P_t)_{t \geq 0}$ be a Markovian convolution semigroup. If the characteristic exponent $\psi$ satisfies the Hartman--Wintner condition \eqref{levy-eq11}, then the kernel $p_t$ satisfies
	\begin{equation}
		\|p_t \mid A_{1,\infty}^u\|
		\leq C_u \left( 1+ \sup_{t/(\floor{u}+1) \leq r \leq t/\floor{u}} \int_{\mbb{R}^n} |\nabla p_r(x)| \, dx \right)^u \label{levy-eq14}
	\end{equation}
	for every $t>0$, $u>0$, $A \in \{B,F\}$ and some constant $C_u>0$. In particular,
	\begin{equation*}
		\|P_t f \mid A_{p,q}^{s+u}\|
		\leq C_u \left( 1+ \sup_{t/(\floor{u}+1) \leq 2r \leq t/\floor{u}} \int_{\mbb{R}^n} |\nabla p_r(x)| \, dx \right)^u \|f \mid A_{p,q}^s\|
	\end{equation*}
	for every $t>0$, $p,q \in [1,\infty]$, $s \in \mbb{R}$ and $u \geq 0$.
\end{corollary}

\begin{proof}
	Since $\psi$ satisfies the Hartman--Wintner condition, there is for each $t>0$ a kernel $p_t$, i.e.\ $P_t f = f*p_t$, with $p_t \in C_b^{\infty}(\mbb{R}^n)$ and $\partial^{\beta} p_t(x) \in L_1$ for all $\beta \in \mbb{N}_0^n$, cf.\ \cite{knop13}. Moreover, $\|p_t \mid L_1\|=1$ since $\mu_t(dy)=p_t(y) \, dy$ is a probability measure. Applying Theorem~\ref{semi-11} proves the assertion.
\end{proof}

Let us illustrate Corollary~\ref{levy-7} with some basic examples.

\begin{example} \label{levy-8}
\begin{enumerate}[wide, labelwidth=!, labelindent=0pt]\renewcommand{\itemsep}{4pt}
\item\label{levy-8-i} (Gau{\ss}--Weierstra{\ss} semigroup)
    Consider $P_t f = f*p_t$ for
	\begin{equation*}
		p_t(x) = \frac{1}{(4\pi t)^{n/2}} \exp \left( - \frac{|x|^2}{4t} \right), \quad x \in \mbb{R}^n,\;t>0.
	\end{equation*}
	From $p_t(x)=t^{-n/2} p_1(x/\sqrt{t})$, we find that
	\begin{equation*}
		\int_{\mbb{R}^n} |\partial_{x_i} p_t(x)| \, dx
		= t^{-(n+1)/2} \int_{\mbb{R}^n} |(\partial_{x_i} p_1)(x/\sqrt{t})| \, dx \\
		= t^{-1/2} \int_{\mbb{R}^n} |\partial_{x_i} p_1(x)| \, dx;
	\end{equation*}
    note that the integral on the right-hand side is finite; in fact, a straightforward computation yields that the integral equals $1/\sqrt{\pi}$. Consequently, Corollary~\ref{levy-7} shows that
	\begin{equation*}
		\|P_t f \mid A_{p,q}^{s+u}\| \leq C (1+t^{-u/2}) \|f \mid A_{p,q}^{s}\|
	\end{equation*}
    for $p,q \in [1,\infty]$, $s \in \mbb{R}$, $u \geq 0$, $t>0$ and $A \in \{B,F\}$, i.e.\ we recover the classical result for the Gau{\ss}--Weierstra{\ss} semigroup, see \cite[Theorem 3.35]{triebel4} and also Corollary~\ref{semi-13}.
	
\item\label{levy-8-ii} (Cauchy--Poisson semigroup)
    Consider $P_t f = f*p_t$ for
	\begin{equation*}
		p_t(x) = \frac{1}{\pi} \frac{t}{x^2+t^2}, \quad x \in \mbb{R},\;t>0,
	\end{equation*}
    As $p_t(x) = t^{-1} p_1(x/t)$, it follows as in the previous example that $\int_{\mbb{R}} |\nabla p_t(x)| \,dx \leq C t^{-1}$. Thus, by Corollary~\ref{levy-7}, the Cauchy--Poisson semigroup satisfies
	\begin{equation*}
		\|P_t f \mid A_{p,q}^{s+u}\| \leq C (1+t^{-u}) \|f \mid A_{p,q}^{s}\|
	\end{equation*}
	for $p,q \in [1,\infty]$, $s \in \mbb{R}$, $u \geq 0$, $t>0$ and $A \in \{B,F\}$.
\end{enumerate}
\end{example}

The two semigroups from Example~\ref{levy-8} are special -- they are among the few Markovian convolution semigroups for which $p_t$  is known explicitly. In general, there is no closed-form for $p_t$, and this makes it much harder -- or even impossible -- to derive estimates for $\int |\nabla p_t(x)| \, dx$ using direct computations. There are, however, different means to establish upper bounds for the integral. Schilling et al.\ \cite{ssw12} showed that the small time asymptotics of $\int |\nabla p_t(x)| \, dx$ is closely related to the growth of (the real part of) the characteristic exponent $\psi$. Combining their result with Corollary~\ref{levy-7}, we get the following statement.

\begin{corollary} \label{levy-9}
    Let $(P_t)_{t \geq 0}$ be a Markovian convolution semigroup. Assume that the characteristic exponent $\psi$ satisfies the Hartman--Wintner condition \eqref{levy-eq11}. If there is a constant $\kappa\in (0,\infty)$ such that
	\begin{equation*}
		\frac 1\kappa \leq \frac{\re \psi(\xi)}{g(|\xi|)} \leq \kappa \quad \text{as $|\xi| \to \infty$}
	\end{equation*}
	for an increasing function $g:(0,\infty) \to (0,\infty)$ which is differentiable near infinity and satisfies
	\begin{equation}
			\liminf_{r \to 0} \frac{g(2 r)}{g(r)} >1, \label{levy-eq18}
	\end{equation}
	then there exists for every $u \geq0$ a constant $C>0$ such that
	\begin{equation*}
		\|P_t f \mid A_{p,q}^{s+u}\| \leq C \left(1 + g^{-1}\left(\frac{ 1}{t}\right) \right)^{ u} \|f \mid A_{p,q}^{s}\|
	\end{equation*}
	for every $p,q \in [1,\infty]$, $s \in \mbb{R}$, $t \in (0,1)$ and $A \in \{B,F\}$.
\end{corollary}

Notice that \eqref{levy-eq18} is equivalent to the following growth condition on the inverse function $g^{-1}$:
	\begin{equation*}
		\limsup_{r \to \infty} \frac{g^{-1}(2r)}{g^{-1}(r)} < \infty,
	\end{equation*}
cf.\ \cite[Lemma 2.2]{deng17}. Corollary~\ref{levy-9} applies to a wide class of Markovian convolution semigroups; in particular, to the semigroups discussed in Example~\ref{levy-8}; in fact, it goes well beyond that as the following example shows.

\begin{example} \label{levy-11}
	Let $(P_t)_{t \geq 0}$ be a Markovian convolution semigroup with characteristic exponent $\psi$ of the form
	\begin{equation*}
		\psi(\xi) = \int_{y \neq 0} (1-e^{iy \cdot \xi} + i y \cdot \xi \I_{(0,1)}(|y|)) \, \nu(dy), \quad \xi \in \mbb{R}^n.
	\end{equation*}
	If the measure $\nu$ satisfies
	\begin{equation*}
		\nu(A)
		\geq \int_0^1 \int_{\mbb{S}^{n-1}} \I_A (r \theta) r^{-1-\alpha} \, \mu(d\theta) \, dr, \quad A \in \mc{B}(\mbb{R}^n \setminus \{0\}),
	\end{equation*}
    for some $\alpha \in (0,2)$ and measure $\mu$ on the sphere $\mbb{S}^{n-1}$ such that $\mu$ is non-degenerate, i.e.\ the support of $ds\otimes d\mu$ is truly $d$-dimensional, then there is for every $u \geq 0$ a constant $C>0$ such that
	\begin{equation*}
		\|P_t f \mid A_{p,q}^{s+u}\| \leq C t^{-u/\alpha} \|f \mid A_{p,q}^s\|, \quad t \in (0,1),
	\end{equation*}
    for all $p,q \in [1,\infty]$, $s \in \mbb{R}$, and $A \in \{B,F\}$. Indeed, if $\nu$ satisfies the above condition, then the assumptions of Corollary~\ref{levy-9} hold with $g(t)=t^{\alpha}$, cf.\ \cite{ssw12}, and the assertion follows.
\end{example}

The above results show that convolution estimates for Markovian convolution semigroups can be obtained under growth conditions on the associated continuous negative definite function $\psi$ or under certain assumptions on the characterizing triplet $(b,Q,\nu)$, cf.\ \eqref{levy-eq5}.

In the concluding part of this section, we use Theorem~\ref{levy-3} to establish estimates for Markovian convolution semigroups which have a particular structure; namely, for subordinate semigroups. Let us briefly recall some basics on Bochner's subordination before stating our result, we refer to the monograph \cite{bernstein} for full details.

A family $(\varrho_t)_{t \geq 0}$ of probability measures on $[0,\infty)$ is called a \emph{vaguely continuous convolution semigroup} if $(\varrho_t)_{t \geq 0}$ is a semigroup w.r.t.\ convolution, that is, $\varrho_{t+\tau}=\varrho_t * \varrho_{\tau}$ for all $t,\tau \geq 0$, and vaguely\footnote{The vague topology is the weak$^*$-topology in the space of (signed) Radon measures $(C_c(\mbb{R}^n))^*$.} continuous in the parameter $t$. Such semigroups can be uniquely characterized by their Laplace transform. More precisely, it holds that
\begin{equation}
	\int_{[0,\infty)} e^{-\lambda r} \, \varrho_t(dr) = e^{-t g(\lambda)}, \quad t \geq 0, \; \lambda \geq 0, \label{levy-eq21}
\end{equation}
for a \emph{Bernstein function} $g$, i.e.\ $g:(0,\infty) \to [0,\infty)$ is smooth and $(-1)^{k-1} g^{(k)} \geq 0$ for all $k \in\mbb{N}$. In fact, there is a one-to-one correspondence between Bernstein functions with $g(0)=0$ and vaguely continuous convolution semigroups of probability measures on $[0,\infty)$.  Important examples of Bernstein functions are $g(\lambda)=\lambda^{\alpha}$, $\alpha \in (0,1)$, $g(\lambda)=\log(1+\lambda)$ and $g(\lambda) = \sqrt{\lambda+c}-\sqrt{c}$, $c \geq 0$; see \cite[Chapter 16]{bernstein} for an extensive list of examples.  Every Bernstein function $g$ has a L\'evy--Khintchine representation
\begin{equation*}
	g(\lambda) = g(0) + b \lambda + \int_{[0,\infty)} (1-e^{-\lambda r}) \, \sigma(dr), \quad \lambda \geq 0,
\end{equation*}
where $b \geq 0$ and $\sigma$ is a measure on $(0,\infty)$ with $\int_{[0,\infty)} \min\{r,1\} \, \sigma(dr)<\infty$. From a probabilist's point of view, $\varrho_t(dr)=\mbb{P}(S_t \in dr)$ is the transition probability of a  \emph{subordinator} $(S_t)_{t \geq 0}$, i.e.\ a L\'evy process on $[0,\infty)$ with non-decreasing sample paths $t \mapsto S_t(\omega)$. \par \medskip

Let $(\varrho_t)_{t \geq 0}$ be a vaguely continuous semigroup of probability measures on $[0,\infty)$ and denote by $g$ the corresponding Bernstein function with $g(0)=0$. If $(P_t)_{t \geq 0}$ is any Markovian convolution semigroup, then
\begin{equation}
	P_t^{(g)} f := \int_{[0,\infty)} P_r f \, \varrho_t(dr) \label{levy-eq23}
\end{equation}
defines a Markovian convolution semigroup. Its characteristic exponent, cf.\ \eqref{levy-eq3}, can be expressed in terms of the characteristic exponent of $(P_t)_{t \geq 0}$ and the Bernstein function $g$. Let us consider the following important example.

\begin{example} \label{levy-12}
    Let $(P_t)_{t \geq 0}$ be the Gau{\ss}--Weierstra{\ss} semigroup, cf.\ Example~\ref{levy-8}\ref{levy-8-i}. By Tonelli's theorem, the subordinate semigroup $(P_t^{(g)})_{t \geq 0}$ satisfies $P_t^{(g)} f = f*\mu_t$ for the measure $\mu_t(dx) := \varrho_t(\{0\}) \delta_0(dx) + p_t(x) \I_{\mbb{R}^n \setminus \{0\}}(x) \, dx$ with
	\begin{equation*}
		p_t(x) = \int_{(0,\infty)} \frac{1}{(4\pi r)^{n/2}} \exp \left(- \frac{|x|^2}{4r} \right) \varrho_t(dr),
        \quad x \in \mbb{R}^n \setminus \{0\}, \; t>0.
	\end{equation*}
	In particular, $\mu_t$ is absolutely continuous with respect to Lebesgue measure if, and only if, $\varrho_t(\{0\})=0$; a sufficient condition is that $g$ satisfies the Hartman-Wintner condition
	\begin{equation*}
		\lim_{\lambda \to \infty} \frac{g(\lambda)}{\log \lambda}= \infty,
	\end{equation*}
	cf.\ \cite[Corollary 3.4]{schoenberg}.
	Moreover, it can be shown that the characteristic exponent $\psi^{(g)}$ of $(P_t^{(g)})_{t \geq 0}$ is $\psi^{(g)}(\xi) = g(|\xi|^2)$. For the particular choice $g(\lambda)=\sqrt{\lambda}$, we thus find that $(P_t^{(g)})_{t \geq 0}$ is the Cauchy--Poisson semigroup, cf.\ Example~\ref{levy-8}\ref{levy-8-ii}. More generally, if $g(\lambda)=\lambda^{\alpha}$ for some $\alpha \in (0,1)$, then $(P_t^{(g)})_{t \geq 0}$ is a stable semigroup; the corresponding infinitesimal generator is the fractional Laplacian $-(-\Delta)^{\alpha}$, which is one of the most studied non-local operators, see \cite{kwas} for a survey. Let us mention that there is a probabilistic interpretation of the subordinate semigroup $(P_t^{(g)})_{t \geq 0}$: If $(B_t)_{t \geq 0}$ is a Brownian motion in $\mbb{R}^n$ and $(S_t)_{t \geq 0}$ is the subordinator associated with $(\varrho_t)_{t \geq 0}$, see above, then $X_t := B_{S_t}$ defines a L\'evy process with distribution $\mbb{P}(X_t \in dx)=\mu_t(dx)$; in particular, we have $P_t^{(g)}f (x) = (f*\mu_t)(x)= \mbb{E}f(x+B_{S_t})$.
\end{example}

In the recent paper \cite{knop19}, the subordination technique was used to establish convolution estimates for a wide class of Markovian convolution semigroups. We recover one of their main results by applying Theorem~\ref{levy-3}.

\begin{corollary} \label{levy-13}
	Let $(T_t)_{t \geq 0}$ be a Markovian convolution semigroup with characteristic exponent $\psi$ of the form $\psi(\xi) = g(|\xi|^2)$, $\xi \in \mbb{R}^n$, for a Bernstein function $g$. If $\psi$ satisfies the Hartman--Wintner condition \eqref{levy-eq11}, then
	\begin{equation}
		\|T_t f \mid A_{p,q}^{s+u}\|
		\leq C_u \left( 1+ \int_{(0,\infty)} r^{-u/2} \, \varrho_t(dr) \right) \|f \mid A_{p,q}^s\| \label{levy-eq25}
	\end{equation}
	for every $p,q \in [1,\infty]$, $s \in \mbb{R}$, $u \geq 0$, $t > 0$, $A \in \{B,F\}$ and some constant $C_u>0$; here $(\varrho_t)_{t \geq 0}$ denotes the vaguely continuous convolution semigroup on $[0,\infty)$ associated with the Bernstein function $g$, cf.\ \eqref{levy-eq21}. If additionally
	\begin{equation*}
		\liminf_{\lambda \to 0} \frac{g(2\lambda)}{g(\lambda)} >1,
	\end{equation*}
	then
	\begin{equation}
		\|T_t f \mid A_{p,q}^{s+u}\|
		\leq C_u' \left( 1+ g^{-1}(1/t)^{u/2} \right) \|f \mid A_{p,q}^s\| \label{levy-eq26}
		\end{equation}
	for every $p,q \in [1,\infty]$, $s \in \mbb{R}$, $u \geq 0$, $t > 0$, $A \in \{B,F\}$ and some constant $C_u'>0$.
\end{corollary}

\begin{remark} \label{levy-15}
\begin{enumerate}[wide, labelwidth=!, labelindent=0pt]
	 \item The integral expression $K_t := \int_{(0,\infty)} r^{-u/2} \, \varrho_t(dr)$ appearing in \eqref{levy-eq25} is finite because $\psi$ satisfies the Hartman--Wintner condition, see \cite[Corollary 3.4]{schoenberg}. The integral $K_t$ can be written more compactly using the notions from probability theory: If $(S_t)_{t \geq 0}$ is the subordinator associated with $(\varrho_t)_{t \geq 0}$, then $\varrho_t(dr) =\mbb{P}(S_t \in dr)$ is the distribution of $S_t$, and so
	\begin{equation*}
		K_t=\int_{(0,\infty)} r^{-u/2} \varrho_t(dr) = \mbb{E}\left(S_t^{-u/2}\right), \quad t>0.
	\end{equation*}
	\item 	In Corollary~\ref{levy-13} we assume that the characteristic exponent $\psi$ is of the form $g(|\xi|^2)$ for a Bernstein function $g$; equivalently, $T_t = P_t^{(g)}$ for the semigroup $(P_t^{(g)})_{t \geq 0}$ obtained from the Gau{\ss}--Weierstra{\ss} semigroup $(P_t)_{t \geq 0}$ by subordination, cf.\ Example~\ref{levy-12}. In fact, our methods can be used to derive convolution estimates for a much larger class of subordinate semigroups. Let $(P_t)_{t \geq 0}$ be any Markovian convolution semigroup with kernels $q_t$ and characteristic exponent $\psi$. If $g$ is a Bernstein function such that $\xi \mapsto g(\psi(\xi))$ satisfies the Hartman--Wintner condition \eqref{levy-eq11}, then the proof of Corollary~\ref{levy-13} shows that the kernels $p_t^{(g)}$ of the subordinate semigroup $(P_t^{(g)})_{t \geq 0}$ satisfy
		\begin{equation*}
			\|p_t^{(g)} \mid A_{1,\infty}^u\| \leq \int_{(0,\infty)} \|q_r \mid A_{1,\infty}^u\| \, \varrho_t(dr),
		\end{equation*}
		and, by Theorem~\ref{levy-3}, this yields convolution estimates for $(P_t^{(g)})_{t \geq 0}$. The norm $\|q_r \mid A_{1,\infty}^u\|$ appearing on the right-hand side can be estimated in many cases, e.g.\ using gradient estimates, see the results presented earlier in this section.
\end{enumerate}
\end{remark}

\begin{proof}[Proof of Corollary~\ref{levy-13}]
	The kernel $p_t$ of $T_t$ is given by
	\begin{equation*}
		p_t(x) = \int_{(0,\infty)} q_r(x) \, \varrho_t(dr), \quad x \in \mbb{R}^n \setminus \{0\},\;t>0,
	\end{equation*}
	where
	\begin{equation*}
		q_r(x) := \frac{1}{(4\pi r)^{n/2}} \exp \left(- \frac{|x|^2}{4r} \right), \quad x \in \mbb{R}^n,\;r>0,
	\end{equation*}
	is the  Gaussian heat kernel, cf.\ Example~\ref{levy-12} Since $\varrho_t$ is a probability measure, the vector-valued triangle inequality shows that
	\begin{equation*}
		\|p_t \mid A_{1,\infty}^u\|
		\leq \int_{(0,\infty)} \|q_r \mid A_{1,\infty}^u\| \, \varrho_t(dr).
	\end{equation*}
	As $\int_{\mbb{R}^n} |\nabla q_r(x)| \, dx \leq c r^{-1/2}$, $r>0$, for some constant $c>0$, cf.\ Example~\ref{levy-8}\ref{levy-8-i}, it follows from Corollary~\ref{levy-7} that $\|q_r \mid A_{1,\infty}^u\| \leq C_u (1+r^{-u/2})$ for every $u>0$. Using that $\varrho_t$ is a probability measure, we conclude that
	\begin{equation*}
		\|p_t \mid A_{1,\infty}^u\|
		\leq C_u \int_{(0,\infty)} (1+r^{-u/2}) \, \varrho_t(dr)
		= C_u \left( 1+ \int_{(0,\infty)} r^{-u/2} \, \varrho_t(dr) \right),
	\end{equation*}
	and by Theorem~\ref{levy-3} this proves \eqref{levy-eq25}. Under the additional growth assumption on $g$, we see that $\int_{(0,\infty)} r^{-u/2} \, \varrho_t(dr)$ can be estimated from above and below by a multiple of $g^{-1}(1/t)^{u/2}$ if $t>0$ is small, cf.\ \cite[Lemma 7.2]{knop19}; thus, \eqref{levy-eq26} follows.
\end{proof}

\begin{ack}
    We thank the editors and the referee for their expert handling of this paper. The referee's comments were most helpful, resulting in an improved version of the paper. We are grateful to Hans Triebel (Jena) for his comments on the reiteration theorem in interpolation theory. The second-named author has been supported through the DFG-NCN Beethoven Classic 3 programme, contract no.~2018/31/G/ST1/02252 (National Science Center, Poland) and SCHI-419/11--1 (DFG, Germany).
\end{ack}

\end{document}